\documentclass[twoside,11pt]{article}

\usepackage{my-jmm}

\begin{document}

\Title{Alternative views on fuzzy numbers and their application\\ to fuzzy differential equations}

\Author{Akbar H. Borzabadi$^\dag$\correspond,  Mohammad Heidari$^{\ddag}$, Delfim F. M. Torres$^{\S}$}

\Address{$^{\dag}$Department of Mathematics, University of Science and Technology
of Mazandaran, Behshahr, Iran\\
$^{\ddag}$Young Researchers and Elite Club, Ayatollah Amoli Branch,
Islamic Azad University, Amol, Iran\\
$^{\S}$Center for Research and Development in Mathematics and Applications (CIDMA),\\
Department of Mathematics, University of Aveiro, 3810-193 Aveiro, Portugal}

\Email{borzabadi@mazust.ac.ir, m.heidari27@gmail.com, delfim@ua.pt}
\Markboth{A. H. Borzabadi, M. Heidari, D. F. M. Torres}{Alternative views on fuzzy numbers ...}

\Abstract{We consider fuzzy valued functions from two parametric
representations of $\alpha$-level sets. New concepts are introduced
and compared with available notions.
Following the two proposed approaches,
we study fuzzy differential equations.
Their relation with Zadeh's extension principle
and the generalized Hukuhara derivative is discussed. Moreover,
we prove existence and uniqueness theorems for fuzzy differential equations.
Illustrative examples are given.}



\section{Introduction}

Fuzzy numbers and fuzzy arithmetic operations were first introduced
in the seventies of the 20th century by Zadeh \cite{Zadeh}
and Dubois and Prade \cite{Dubois}, respectively. Fuzzy set theory
is nowadays a powerful mathematical tool for modeling uncertainty
and processing vague or subjective information \cite{klir,Melliani}.
The field has been developed in several major directions and find
applications in many different real-world problems. Roughly speaking,
two approaches to fuzzy arithmetic are available: one based on Zadeh's
extension principle \cite{Bedeebook}, the other based on
interval arithmetic \cite{Goetschel}.

One of the major applications of fuzzy arithmetic is given
by fuzzy differential equations (FDEs).
Indeed, FDEs provide a natural way to model dynamic systems
under uncertainty \cite{Zadeh2}, and have shown to be an effective and useful
technique in a large number of different areas, such as in civil engineering
\cite{Oberguggenberger}, hydraulics \cite{Bencsik},
and population dynamics \cite{Barros}.
There are at least three approaches to FDEs:
the first based on Zadeh's extension principle, a second one on families
of differential inclusions, and a third using some appropriate
derivative for fuzzy valued functions. Under certain conditions,
the three approaches are equivalent, although in general they differ
\cite{Chalco2008,Chalco2009,Chalco2007}. Different ways can be also used
to realize each one of the three approaches.

One of the first approaches to FDEs via Zadeh's extension principle
was proposed by Buckley and Feuring \cite{Buckley}, where a solution
to a FDE is generated using Zadeh's extension principle on a
classical solution of its ordinary differential equation.
A second way has been proposed by Barros et al.: to use Zadeh's extension principle
to define a derivative for fuzzy valued functions and then use it to develop the
corresponding theory of fuzzy dynamical systems \cite{Barros}. In this approach,
the utilization of Zadeh's extension principle is, in practice, quite difficult,
and different formulations of the same ordinary differential equation lead to the same solution,
which ensures its uniqueness. To overcome the difficulties, a different approach to FDEs was suggested
by H\"{u}llermeier \cite{Hullermeier} (see also \cite{Chalco2013}), by taking the $\alpha$-level
sets of the parameters, the initial value and the solution, and converting the given differential equation into
a family of differential inclusions. Other approaches exist, e.g., based on the definition of fuzzy differences.
Then, different ways to study FDEs correspond to many different suggestions given in the literature
on how to define a fuzzy difference.

Initially, Kaleva \cite{Kaleva} formulated the concept of solution to a FDE by using differentiability
in the sense of the Hukuhara difference, which was first introduced by Puri and Ralescu \cite{Puri}.
It is well-known that the usual Hukuhara difference between two fuzzy numbers
exists only under very restrictive conditions. For details, see \cite{Diamond,Diamond2,Kaleva}.
To overcome this shortcoming, Stefanini and Bede proposed the generalized Hukuhara
difference of two fuzzy numbers, which exists in many more situations than the Hukuhara difference
\cite{Stefanini,Stefanini2}. The same authors extended this line of research
by introducing the so-called generalized difference, which has a large
advantage over previous concepts, namely, it always exists \cite{Bede2013,Stefanini}.
Independently, in 2015, Gomes and Barros published a note on the generalize difference
in order to assure existence \cite{Gomes2015}. The literature on fuzzy differentiability concepts
is now vast \cite{Bedeebook,Bede2005,Bede2013}. In particular, based on the notion of strong generalized
differentiability, which follows from the Hukuhara difference and generalized Hukuhara differentiability,
FDEs have been extensively investigated and several results on existence and uniqueness of solutions
are available in the literature \cite{Bede2005,Bede2010,Chalco2008,Lupulescu}.

Here we propose new differentiability and integrability concepts for fuzzy
valued functions. Several properties of the new concepts are investigated
and compared with similar concepts that have been recently introduced.
Using the new notions for fuzzy derivatives and integrals,
Newton--Leibniz type formulas are non-trivially extended to the fuzzy case.
Then, using the obtained results, FDEs are investigated
under two different approaches. In the first approach, a fuzzy problem
is transformed into a crisp one. It is shown, under suitable assumptions,
that this approach reproduces the same solutions as those obtained
via differential inclusions and Zadeh's extension principle. Our second approach
follows from the proposed notion of derivative for a fuzzy valued function. 
We show that fuzzy solutions obtained from this approach coincide with those 
obtained via the generalized Hukuhara differentiability concept.

The paper is organized as follows. In Section~\ref{sec:02},
parametric representations of fuzzy numbers and their main properties are recalled.
Our definitions of fuzzy derivative and integral for fuzzy valued functions,
and their relation and comparison with other recent proposals, is given in Section~\ref{sec:03}.
Section~\ref{sec:04} is dedicated to the study of fuzzy differential equations.
The paper ends with conclusions, in Section~\ref{sec:05}.


\section{Preliminaries}
\label{sec:02}

There are several models to obtain parametric representations of fuzzy numbers
and their arithmetic operators: see, e.g., \cite{Chalco20,Giachetti,Lodwick,Stefaninia2006}.
In \cite{Chalco20}, Chalco-Cano et al. considered two parametric representations
for each fuzzy number $u$ as follows:
\begin{itemize}
\item[(i)] the decreasing convex constraint function
$u: [0,1]^2\rightarrow \mathbb{R}$ is defined by
\begin{equation}
\label{eqq1}
u_{\alpha}(\lambda)=\lambda u^-_{\alpha}+(1-\lambda)u^+_{\alpha},
\quad 0 \leq \lambda, \alpha \leq 1;
\end{equation}

\item[(ii)] and the increasing convex constraint
function $u: [0,1]^2\rightarrow \mathbb{R}$ by
\begin{equation}
\label{eqq2}
u_{\alpha}(\lambda)=\lambda u^+_{\alpha}+(1-\lambda)u^-_{\alpha},
\quad 0\leq \lambda, \alpha \leq 1.
\end{equation}
\end{itemize}
Based on the parametric representations \eqref{eqq1} and \eqref{eqq2},
a variant of the single-level constraint interval arithmetic (SLCIA)
was proposed, in order to exhibit the operations between fuzzy numbers, as
$$
[u]_{\alpha} \circledast [v]_{\alpha}=\left[\min_{0\leq \lambda
\leq 1}(u_{\alpha}(\lambda)\ast v_{\alpha}(\lambda)),
\max_{0\leq \lambda \leq 1}(u_{\alpha}(\lambda)\ast v_{\alpha}(\lambda))\right],
$$
where $\ast \in \{+,-,\times, \div\}$. Unlike SLCIA, Lodwick and Dubois \cite{Lodwick}
introduced the constraint interval arithmetic (CIA) to present another
fuzzy arithmetic, using a distinct parameter $\lambda$ for each distinct
interval involved in operations, as
\begin{multline*}
[u]_{\alpha} \circledast [v]_{\alpha}
=\Big{[}\min_{0\leq \lambda_u,\lambda_v \leq 1}\left\{(u^-_{\alpha}
+\lambda_u(u^+_{\alpha}-u^-_{\alpha}))\ast
\left(v^-_{\alpha}+\lambda_v(v^+_{\alpha}-v^-_{\alpha})\right) \right\}, \\
\max_{0\leq \lambda_u,\lambda_v \leq 1}
\left\{(u^-_{\alpha}+\lambda_u(u^+_{\alpha}-u^-_{\alpha}))
\ast (v^-_{\alpha}+\lambda_v(v^+_{\alpha}-v^-_{\alpha})) \right\} \Big{]}.
\end{multline*}
SLCIA and CIA have some remarkable properties \cite{Chalco20,Lodwick}.
Recently, Heidari et al. \cite{Heidari2017,Heidari2016}, based on the parametric
representations \eqref{eqq1} and \eqref{eqq2}, found the solutions
of fuzzy variational and unconstrained fuzzy valued optimization problems.
Moreover, Ramezanadeh et al., by considering these parametric representations
for interval numbers, investigated interval differential
equations with two different approaches \cite{Ramezanadeh}.
As an extension of the proposed approach in \cite{Ramezanadeh}, here
the $\alpha$-level set of fuzzy valued functions is expressed as a set of classical functions
using the parametric representations \eqref{eqq1} and \eqref{eqq2}.


\subsection{Basic concepts of fuzzy set theory}

We denote by $\mathcal{F}(\mathbb{R})$ the set of fuzzy numbers, i.e., normal, fuzzy convex, upper
semicontinuous and compactly supported fuzzy sets defined over the real line $\mathbb{R}$.
For $0<\alpha\leq 1$, the $\alpha$-level set of $A \in \mathcal{F}(\mathbb{R})$
is defined by $[A]_{\alpha} = \{x \in \mathbb{R}| A(x) \geq \alpha\}$ and for $\alpha=0$
it is the closure of the support, i.e., $[A]_{0} =\overline{\{x \in \mathbb{R}| A(x) > 0\}}$.
The relationship between crisp sets and fuzzy sets can be obtained from the following theorems and lemma.

\begin{theorem}[Stacking theorem, Negoi\c t\u a and Ralescu \cite{Negoita}]
\label{th22}
A fuzzy number $A$ satisfies the following conditions:\\
$(1)$ its $\alpha$-level sets are nonempty closed intervals for all $\alpha\in[0, 1]$;\\
$(2)$ $[A]_{\alpha} \subseteq [A]_{\beta}$ for all $0 \leq \beta \leq \alpha \leq 1$;\\
$(3)$ $\cap_{i=1}^{\infty} [A]_{\alpha_i}=[A]_{\alpha}$ for any sequence
$\alpha_i$ that converges from below to $\alpha \in ]0, 1]$;\\
$(4)$ $\overline{\cup_{i=1}^{\infty} [A]_{\alpha_i}}=[A]_0$
for any sequence $\alpha_i$ that converges from above to $0$.
\end{theorem}

\begin{theorem}[Characterization theorem, Negoi\c t\u a and Ralescu \cite{Negoita}]
\label{th2}
Let $\{ U_{\alpha}|~ 0\leq \alpha \leq 1 \}$ be a
family of subsets $\mathbb{R}$ satisfying the following conditions:\\
$(1)$ $U_{\alpha}$ are nonempty closed intervals for all $\alpha\in [0, 1]$;\\
$(2)$ $U_{\alpha} \subseteq U_{\beta}$ for all $0 \leq \beta \leq \alpha \leq 1$;\\
$(3)$ $\cap_{i=1}^{\infty} U_{\alpha_i}=U_{\alpha}$ for any sequence
$\alpha_i$ that converges from below to $\alpha \in ]0, 1]$;\\
$(4)$ $\overline{\cup_{i=1}^{\infty} U_{\alpha_i}}=U_0$ for any sequence
$\alpha_i$ that converges from above to $0$.\\
Then, there exists a unique fuzzy number $A$
such that $[A]_{\alpha} = U_{\alpha}$ for any $\alpha\in[0,1]$.
\end{theorem}

The notation $[A]_{\alpha}=[a^-_{\alpha}, a^+_{\alpha}]$ denotes explicitly
the $\alpha$-level set of a fuzzy number $A$,
where $a^-_{\alpha}$ and $a^+_{\alpha}$ are its lower and upper bounds, respectively.
The following well-known result represents some interesting properties associated to
$a^-_{\alpha}$ and $a^+_{\alpha}$ of a fuzzy number $A$.

\begin{lemma}[See \cite{Bedeebook}]
\label{lem1}
Assume that $a^-: [0,1]\rightarrow \mathbb{R}$ and
$a^+: [0,1] \rightarrow \mathbb{R}$ satisfy the following conditions:\\
$(1)$ $a^-(\alpha)=a^-_{\alpha} \in \mathbb{R}$ is a bounded, non-decreasing,
left-continuous function in $]0, 1]$ and it is right-continuous at $0$;\\
$(2)$ $a^+(\alpha)=a^+_{\alpha} \in \mathbb{R}$ is a bounded, non-increasing,
left-continuous function in $]0, 1]$ and it is right-continuous at $0$;\\
$(3)$ $a^-_1 \leq a^+_1$ for $\alpha=1$, which implies $a^-_{\alpha} \leq a^+_{\alpha}$
for all $\alpha \in [0, 1]$.\\
Then, there is a fuzzy number $A\in\mathcal{F}(\mathbb{R})$ that has
$a^-_{\alpha}, a^+_{\alpha}$ as endpoints of its $\alpha$-level set
$[A]_{\alpha}$. Moreover, if $A:\mathbb{R} \rightarrow [0,1]$ is a fuzzy number
with $[A]_{\alpha}=[a^-_{\alpha}, a^+_{\alpha}]$, then functions $a^-_{\alpha}$
and $a^+_{\alpha}$ satisfy the above conditions $(1)$--$(3)$.
\end{lemma}


\subsection{Parametric representation of fuzzy numbers}

The $\alpha$-level set $[A]_{\alpha}=[a^-_{\alpha}, a^+_{\alpha}]$
allows us to consider non-decreasing and non-increasing parametric
representations as follows \cite{Chalco20}:
\begin{align}
\label{eq1}
[A]_{\alpha}&=\{ a(t,\alpha)|~a(t, \alpha)=a^-_{\alpha} + t(a^+_{\alpha} - a^-_{\alpha});
~~t \in [0, 1] \} ~\text{(non-decreasing representation)},\\\label{eq2}
[A]_{\alpha}&=\{ a(t,\alpha)|~a(t, \alpha)=a^+_{\alpha} + t(a^-_{\alpha} - a^+_{\alpha});
~~t \in [0, 1] \} ~\text{(non-increasing representation)}.
\end{align}
Using these parametric representations, one defines the equality relation of two fuzzy numbers.

\begin{definition}
\label{def1}
We write $A=B\in\mathcal{F}(\mathbb{R})$ if and only if $[A]_{\alpha}=[B]_{\alpha}$
for all $\alpha\in[0,1]$. In other words,  $A=B$ if and only if
$\{ a(t,\alpha)|~ t \in [0, 1] \}=\{ b(t,\alpha)|~ t \in [0, 1] \}$,
where $a(t,\alpha)$ and $b(t,\alpha)$ are non-decreasing (non-increasing)
parametric representations of $[A]_{\alpha}$ and $[B]_{\alpha}$, respectively.
\end{definition}

\begin{remark}
\label{rem1}
Following Definition~\ref{def1}, $A=B$ if and only if
$a(t,\alpha)=b(t,\alpha)$ for all $t\in[0,1]$.
\end{remark}

Based on the parametric representations \eqref{eq1} and \eqref{eq2}, the standard
interval arithmetic is explicitly extended to the fuzzy case.

\begin{definition}
\label{def2}
Let $[A]_{\alpha}=\{ a(t,\alpha)|~t \in [0, 1] \}$ and $[B]_{\alpha}=\{ b(t,\alpha)|~t \in [0, 1] \}$
be the non-decreasing (non-increasing) representations of $\alpha$-level sets $A, B\in\mathcal{F}(\mathbb{R})$,
respectively, and $\lambda$ be a real number. The parametric arithmetic $A*B$ and $\lambda \cdot A$
are defined in terms of their $\alpha$-level sets:
\begin{align}
\label{eqqq1}
[A * B]_{\alpha} &= \{ a(t_1, \alpha)*b(t_2, \alpha)| ~t_1, t_2 \in [0, 1]\},\\\label{eqqq2}
[\lambda \cdot A]_{\alpha} &= \{ \lambda  a(t, \alpha)| ~t \in [0, 1]\},
\end{align}
where $[A]_{\alpha}=[a^-_{\alpha},a^+_{\alpha}]$, $[B]_{\alpha}=[b^-_{\alpha},b^+_{\alpha}]$,
and $0\notin[B]_0$ in the division $A \div B$.
\end{definition}

\begin{remark}
\label{rem11}
Note that if $[A]_{\alpha}$ has a non-decreasing (non-increasing) parametric representation
and $\lambda \geq 0$, then $[\lambda \cdot A]_{\alpha}$ has a non-decreasing (non-increasing)
parametric representation, and if $\lambda < 0$, then $[\lambda \cdot A]_{\alpha}$
has a non-increasing (non-decreasing) parametric representation.
\end{remark}

Because, for each $\alpha\in[0,1]$, $a(t_1,\alpha)*b(t_2,\alpha)$
and $a(t,\alpha)$ are continuous functions in $t_1, t_2$ and $t$,
respectively, one can easily check that the parametric
arithmetic defined in Definition \ref{def2} is equivalent to the
fuzzy standard interval arithmetic of $A*B$ and $\lambda \cdot A$ \cite{Klir2},
which is defined via their $\alpha$-level sets by
\begin{equation*}
[A*B]_{\alpha}=[A]_{\alpha}*[B]_{\alpha},
~[\lambda\cdot A]_{\alpha}=\lambda[A]_{\alpha}
~~\text{for all} ~\alpha\in[0,1],
\end{equation*}
where
\begin{equation*}
[A]_{\alpha}*[B]_{\alpha}=\{c|~ c=a*b, a\in[A]_{\alpha}, b\in[B]_{\alpha}\},
~\lambda[A]_{\alpha}=\{\lambda a|~ a\in[A]_{\alpha}\}.
\end{equation*}
Parametric arithmetic \eqref{eqqq1},\eqref{eqqq2} and CIA are identical, except
that when the fuzzy numbers are the same, parametric arithmetic \eqref{eqqq1}
uses two parameters, while CIA uses one parameter, i.e.,
\begin{align*}
[A\ast A]_{\alpha}&=\left\{ a(t,\alpha)\ast a(t,\alpha)|~t\in [0,1] \right\}
\quad \text{in CIA},\\
[A\ast A]_{\alpha}&=\left\{ a(t_1,\alpha)\ast a(t_2,\alpha)|~t_1, t_2\in [0,1] \right\}
\quad \text{in parametric arithmetic \eqref{eqqq1}}.
\end{align*}
It is noteworthy that, by considering a single parameter for all fuzzy operands in the
parametric arithmetic, that is, $t_1=t_2=t$, the SLCIA is obtained.

In agreement with Definition~\ref{def2}, the difference has the property $A-A\neq 0$.
To overcome this issue, Chalco-Cano et al. \cite{Chalco20} proposed C-subtraction (using SLCIA),
which is equivalent to the generalized Hukuhara difference, as follows.

\begin{definition}[Chalco-Cano et al. \cite{Chalco20}]
\label{def3}
The parametric difference ($p$-difference for short) of two fuzzy numbers
$A, B \in \mathcal{F}(\mathbb{R})$ is given by its $\alpha$-level set as
\begin{equation*}
[A \ominus_p B]_{\alpha} = \{ a(t, \alpha)-b(t, \alpha)|
~a(t, \alpha)=a^-_{\alpha} + t(a^+_{\alpha} - a^-_{\alpha}), ~b(t, \alpha)
=b^-_{\alpha} + t(b^+_{\alpha} - b^-_{\alpha});~~t \in [0, 1]\},
\end{equation*}
where $[A]_{\alpha}=[a^-_{\alpha},a^+_{\alpha}]$
and $[B]_{\alpha}=[b^-_{\alpha},b^+_{\alpha}]$.
\end{definition}

If in Definition~\ref{def3}, the non-decreasing (non-increasing) parametric representation
for the fuzzy number $A$ and the non-increasing (non-decreasing) parametric representation
for the fuzzy number $B$ are considered, then $A\ominus_p B = A-B$.

\begin{remark}
\label{rem2}
If $a(t, \alpha)-b(t, \alpha)$ is a non-decreasing (non-increasing)
function in $t$ for all $\alpha \in [0,1]$, then the
parametric representation of $[A \ominus_p B]_{\alpha}$
is according with the non-decreasing (non-increasing) representation.
\end{remark}

\begin{remark}
\label{rem22}
Clearly, if $u\in\mathcal{F}(\mathbb{R})$ and $v\in\mathbb{R}$,
then $u\ominus_pv=u-v$ and $v\ominus_pu=v-u$.
\end{remark}

Let $h=h(t, \alpha, x)$ be a real valued function. Throughout the
paper, the following notations are used: $h'=\frac{\partial
h}{\partial x}$, $\Delta_th=\frac{\partial h}{\partial t}$ and
$\Delta_{\alpha}h=\frac{\partial h}{\partial \alpha}$. Moreover, it
is assumed that the following equalities hold for the mixed
derivatives:
\begin{align*}
\left(\Delta_{\alpha}h\right)'
&=\frac{\partial}{\partial x}\left(\frac{\partial h}{\partial \alpha}\right)
=\frac{\partial}{\partial \alpha}\left(\frac{\partial h}{\partial x}\right)=\Delta_{\alpha}(h'),\\
\left(\Delta_th\right)'
&=\frac{\partial}{\partial x}\left(\frac{\partial h}{\partial t}\right)
=\frac{\partial}{\partial t}\left(\frac{\partial h}{\partial x}\right)=\Delta_t(h').
\end{align*}

Using the parametric representation,
the obtained results by Bede and Stefanini \cite{Bede2013}
and Chalco-Cano et al. \cite{Chalco200} are easily achieved.
This follows from Proposition~\ref{prop01}.

\begin{proposition}
\label{prop01}
Let $A, B\in \mathcal{F}(\mathbb{R})$ be defined in terms of their $\alpha$-level sets
$\{a(t, \alpha)| ~a(t, \alpha)=a^-_{\alpha} + t(a^+_{\alpha} - a^-_{\alpha}); ~t \in[0,1]\}$,
$\{b(t, \alpha)| ~b(t, \alpha)=b^-_{\alpha} + t(b^+_{\alpha} - b^-_{\alpha}); ~t \in [0, 1]\}$, respectively.
Then, there exists a fuzzy number $C\in \mathbb{R}_\mathcal{F}$ such that $C=A\ominus_pB$
if and only if one of the following conditions is satisfied:
\begin{equation}\label{equ5}
\left
\{
\begin{array}{ll}
\Delta_{\alpha}c(0,\alpha)\geq0,\\
\Delta_{\alpha}c(1,\alpha)\leq0, ~~\forall \alpha\in[0,1]\\
\Delta_tc(t,1)\geq0,
\end{array}
\right.
\end{equation}
or
\begin{equation}
\label{equ6}
\left
\{
\begin{array}{ll}
\Delta_{\alpha}c(0,\alpha)\leq0,\\
\Delta_{\alpha}c(1,\alpha)\geq0, ~~\forall \alpha\in[0,1],\\
\Delta_tc(t,1)\leq0,
\end{array}
\right.
\end{equation}
where $c(t,\alpha)=a(t,\alpha)-b(t,\alpha)$.
\end{proposition}

\begin{proof}
It can be easily proved from Lemma~\ref{lem1}.
\end{proof}

In Proposition~\ref{prop01}, if both conditions \eqref{equ5}
and \eqref{equ6} hold, then $A\ominus_pB$ is a crisp number.

The $p$-difference of two fuzzy numbers does not always exist.
For example, consider the two triangular fuzzy numbers $A=(12,15,19)$
and $B=(5,9,11)$. The $p$-difference $A\ominus_pB$
does not exists because
$$
[A\ominus_pB]_{\alpha}=\{ 7-\alpha+t(1-\alpha)|~ t \in[0,1] \}
$$
does not satisfy the conditions \eqref{equ5} and \eqref{equ6} of Proposition~\ref{prop01}.
To overcome this shortcoming, a new difference between fuzzy numbers has been defined,
that always exists (cf. Theorem~4 and Remark~7 in \cite{Chalco20}).

\begin{definition}
\label{def33}
The generalized parametric difference ($gp$-difference for short) of two fuzzy numbers
$A, B \in \mathcal{F}(\mathbb{R})$ is given by its $\alpha$-level set as
\begin{align*}
[A \ominus_{gp} B]_{\alpha} &= \left[\inf_{\beta\geq\alpha}
\min_{t} (a(t, \beta)-b(t, \beta)), \sup_{\beta\geq\alpha}
\max_{t}(a(t, \beta)-b(t, \beta))\right],
\end{align*}
where $a(t, \beta)=a^-_{\beta} + t(a^+_{\beta} - a^-_{\beta})$
and $b(t, \beta)=b^-_{\beta} + t(b^+_{\beta} - b^-_{\beta})$.
\end{definition}

\begin{proposition}[See \cite{Chalco20}]
For any fuzzy numbers $A, B\in \mathcal{F}(\mathbb{R})$,
the $gp$-difference $A\ominus_{gp}B$ exists and is a fuzzy number.
Moreover, the $gp$-difference and the generalized difference coincide.
\end{proposition}

\begin{example}
Consider the triangular fuzzy numbers $A=(12,15,19)$, $B=(5,9,11)$, $C=(0,5,10)$,
and the trapezoidal fuzzy number $D=(4,5,6,8)$. It is easy to see that the
$p$-differences $A\ominus_pB$ and $D\ominus_pC$ do not exist,
while their $gp$-differences exist:
\begin{align*}
[A\ominus_{gp}B]_{\alpha}&=\left[\inf_{\beta\geq\alpha}\min_{t} (7-\beta+t(1-\beta)),
\sup_{\beta\geq\alpha}\max_{t} (7-\beta+t(1-\beta))\right]\\
&=\left[\inf_{\beta\geq\alpha} (7-\beta), \sup_{\beta\geq\alpha}(8-2\beta)\right]=[6,8-2\alpha],\\
[D\ominus_{gp}C]_{\alpha}&=\left[\inf_{\beta\geq\alpha}\min_{t} (4-4\beta+t(-6+7\beta)),
\sup_{\beta\geq\alpha}\max_{t} (4-4\beta+t(-6+7\beta))\right]\\
&=\left\{
\begin{array}{ll}
\!\Big{[}\underset{\beta\geq\alpha}\inf (3\beta-2),
\underset{\beta\geq\alpha}\sup(4-4\beta)\Big{]},
& \beta\in[0,\frac{6}{7}]\\
\!\Big{[}\underset{\beta\geq\alpha}\inf (4-4\beta),
\underset{\beta\geq\alpha}\sup(3\beta-2)\Big{]},
& \beta\in[\frac{6}{7},1]
\end{array}
\right.\\
&=\left\{
\begin{array}{ll}
\![3\alpha-2, 4-4\alpha],& \alpha\in[0,\frac{2}{3}]\\
\![0, 4-4\alpha], &\alpha\in[\frac{2}{3},\frac{3}{4}]\\
\![0,1], &\alpha\in[\frac{3}{4},1].
\end{array}
\right.
\end{align*}
\end{example}

\begin{definition}
\label{def333}
Let $A, B\in \mathcal{F}(\mathbb{R})$ be defined in terms of their $\alpha$-level sets
$\{a(t, \alpha)| ~a(t, \alpha)=a^-_{\alpha} + t(a^+_{\alpha} - a^-_{\alpha}); ~t \in[0,1]\}$,
$\{b(t, \alpha)| ~b(t, \alpha)=b^-_{\alpha} + t(b^+_{\alpha} - b^-_{\alpha}); ~t \in [0, 1]\}$,
respectively. The metric $D: \mathcal{F}(\mathbb{R})\times \mathcal{F}(\mathbb{R})
\rightarrow \mathbb{R}_+\cup\{0\}$ is defined as
$$
D(A, B) = \underset { 0 \leq \alpha \leq 1 }{ \sup }
\left\{ d([A]_{\alpha}, [B]_{\alpha}) \right\},
$$
where
\begin{equation}
\label{equ4}
d([A]_{\alpha}, [B]_{\alpha})
= \underset{t}{\max}|a(t, \alpha)-b(t, \alpha)|.
\end{equation}
\end{definition}

Relation \eqref{equ4} is equivalent to the Hausdorff distance
between $[A]_{\alpha}$ and $[B]_{\alpha}$.
Let $k\in\mathbb{R}$ and $A, B, C, E \in\mathcal{F}(\mathbb{R})$.
The following properties are well-known for the metric $D$:
\begin{itemize}
\item $D(A+C,B+C)=D(A,B)$,
\item $D(kA, kB)=|k|D(A, B)$,
\item $D(A+B, C+E)\leq D(A, C)+D(B, E)$,
\item $(\mathcal{F}(\mathbb{R}), D)$ is a complete metric space.
\end{itemize}

\begin{proposition}
\label{prop11}
For $A,B\in\mathcal{F}(\mathbb{R})$, if $A\ominus_pB$ exists, then
\begin{equation*}
D(A,B)=D(A\ominus_pB,0).
\end{equation*}
\end{proposition}

\begin{proof}
It is easily proved from Definition~\ref{def333}.
\end{proof}

An immediate property, that follows from Proposition~\ref{prop11}, is
\begin{equation*}
D(A,B)=0 \Leftrightarrow A\ominus_pB=0.
\end{equation*}


\section{Main results: fuzzy valued functions}
\label{sec:03}

Let $F:{\cal S}\subseteq \mathbb{R}\rightarrow
\mathcal{F}(\mathbb{R})$ be a fuzzy valued function with
$\alpha$-level set $[F(x)]_{\alpha}=[f^-_{\alpha}(x),
f^+_{\alpha}(x)]$. Then, it is trivial to see that a parametric
representation of $[F(x)]_{\alpha}$ is
\begin{equation}
\label{equ1}
\left\{f_{(t, \alpha)}(x)=f^-_{\alpha}(x)
+t(f^+_{\alpha}(x)-f^-_{\alpha}(x))| ~t\in[0,1] \right\}.
\end{equation}

\begin{definition}[See \cite{Heidari2017}]
\label{def6} The fuzzy number $L$ is the limit of the fuzzy valued
function $F: {\cal S}\subseteq\mathbb{R}\rightarrow
\mathcal{F}(\mathbb{R})$ as $x\rightarrow a$ if, and only if, for
every $\epsilon>0$ there exists $\delta>0$ such that $D(F(x),
L)<\epsilon,$ whenever $| x-a |<\delta$.
\end{definition}

\begin{proposition}
\label{prop3} For a fuzzy valued function ${F: {\cal
S}\subseteq\mathbb{R}\rightarrow \mathcal{F}(\mathbb{R})}$,
\begin{equation*}
\underset{x\rightarrow x_0}{\lim} F(x)=L \Leftrightarrow
\underset{x\rightarrow x_0}{\lim} (F(x)\ominus_p L)=0.
\end{equation*}
\end{proposition}

\begin{proof}
It is easily obtained from Definition~\ref{def6} and Proposition~\ref{prop11}.
\end{proof}

\begin{definition}
\label{def7} Let $F: {\cal S}\subseteq\mathbb{R}\rightarrow
\mathcal{F}(\mathbb{R})$ be a fuzzy valued function. We say that $F$
is continuous at $a\in {\cal S}$ if and only if for every
$\epsilon>0$, there exists a $\delta=\delta(a, \epsilon)>0$ such
that
\begin{equation*}
D(F(x), F(a))<\epsilon
\end{equation*}
for all $x\in {\cal S}$ with $| x-a |<\delta$, that is,
\begin{equation*}
\underset{x\rightarrow a}{\lim} F(x)=F(a).
\end{equation*}
Moreover, we say that $F(x)$ is continuous if it is continuous at
any point in ${\cal S}$.
\end{definition}

\begin{definition}[See \cite{Zadeh1994}]
The Zadeh's extension $F: \mathcal{F}(U) \rightarrow \mathcal{F}(V)$
of a function $f: U \rightarrow V$, where $U$ and $V$ are topological spaces,
is defined as follows: given $u \in \mathcal{F}(U)$,
\begin{equation*}
\mu_{F(u)}(y) =
\left\{
\begin{array}{ll}
\underset { s \in f^{-1}(y) }{ \sup } \mu_u(s) ~~&\text{if} ~~~f^{-1}(y)\neq \emptyset\\
\quad 0 &\text{if} ~~~f^{-1}(y) = \emptyset
\end{array}
\right.
\end{equation*}
for all $y \in V$.
\end{definition}

In general, the computation of Zadeh's extension principle is a rather difficult task.
Simplicity is found, however, if the function to be extended is continuous \cite{ChalcoY}.

\begin{theorem}[See \cite{Huang}]
\label{th1}
If $f: \mathbb{R}^n \rightarrow \mathbb{R}^m$ is continuous,
then Zadeh's extension $F:\mathcal{F}(\mathbb{R}^n) \rightarrow \mathcal{F}(\mathbb{R}^m)$
is well defined, continuous, and $[F(u)]_{\alpha} = f([u]_{\alpha})$ for all $\alpha \in [0, 1]$.
\end{theorem}

Let $\textbf{C}_{\nu}^k \in (\mathcal{F}(\mathbb{R}))^k$ be a
$k$-dimensional fuzzy vector for which each element is a fuzzy
number, i.e., $\textbf{C}_{\nu}^k = (C_1, C_2, \cdots, C_k)^T,~C_j
\in \mathcal{F}(\mathbb{R})$, $j=1, \cdots, k$, which denotes the
set of all parameters that are present in a fuzzy valued function.
Without loss of generality, consider $\textbf{C}_{\nu}^k$ to be an
ordered set with respect to (w.r.t.) the order maintained in the
function. If the fuzzy valued function $F_{\textbf{C}_{\nu}^k}:
{\cal S}\subseteq\mathbb{R} \rightarrow \mathcal{F}(\mathbb{R})$ is
obtained from a continuous function by applying Zadeh's extension
principle, then, using a non-decreasing parametric representation of
the $\alpha$-level sets for the fuzzy numbers and Theorem~\ref{th1},
another parametric representation for the $\alpha$-level set of this
class of fuzzy valued functions is obtained:
\begin{equation}
\label{eq3} [F_{\textbf{C}_{\nu}^k}(x)]_{\alpha}=\left\{
f_{\textbf{c}(t, \alpha)}(x)\big| f_{\textbf{c}(t, \alpha)}:{\cal
S}\subseteq\mathbb{R}\rightarrow \mathbb{R}; \textbf{c}(\textbf{t},
\alpha)\in [\textbf{C}_{\nu}^k]_{\alpha} \right\},
\end{equation}
where the $\alpha$-level set of the fuzzy vector $\textbf{C}_{\nu}^k$ is defined as
\begin{align}
\nonumber
[\textbf{C}_{\nu}^k]_{\alpha} = \{\textbf{c}(\textbf{t}, \alpha)|~\textbf{c}(\textbf{t}, \alpha)
=(c_1(t_1, \alpha),~c_2(t_2, \alpha),&\cdots, c_k(t_k, \alpha))^T; ~c_i(t, \alpha)
=(c_i)^-_{\alpha} + t_i((c_i)^+_{\alpha} - (c_i)^-_{\alpha}),\\\label{eq22}
&i=1, \cdots, k,~\textbf{t}=(t_1, t_2, \cdots, t_k)\in[0,1]^k\}.
\end{align}
By \eqref{equ1}, since for every fixed $x$ and $\alpha \in [0,1]$,
$f_{\textbf{c}(t, \alpha)}(x)$ is linear with respect to
$\textbf{t}$ and therefore continuous, $\min_{\textbf{c}(\textbf{t},
\alpha) \in [\textbf{C}_{\nu}^k]_{\alpha}}
{f_{\textbf{c}(\textbf{t}, \alpha)}(x)} =\min_{\textbf{t}\in [0,
1]^k} f_{\textbf{c}(\textbf{t}, \alpha)}(x)$ and
$\max_{\textbf{c}(\textbf{t}, \alpha)\in
[\textbf{C}_{\nu}^k]_{\alpha}} {f_{\textbf{c}(\textbf{t},
\alpha)}(x)}= \max_{\textbf{t}\in [0, 1]^k}
f_{\textbf{c}(\textbf{t}, \alpha)}(x)$ exist, then
\begin{equation*}
[F_{\textbf{C}_{\nu}^k}(x)]_{\alpha} = \left[\min_{\textbf{t}\in [0, 1]^k}
f_{\textbf{c}(\textbf{t}, \alpha)}(x), \max_{\textbf{t}\in [0, 1]^k}
f_{\textbf{c}(\textbf{t}, \alpha)}(x)\right].
\end{equation*}
Note that the parametric variables $t_j$ in \eqref{eq3}, $j=1, 2,
\cdots, k$, belong to the interval $[0,1]$. However, the parameter
$t$ in \eqref{equ1} depends on $x$, i.e., $t:{\cal
S}\subseteq\mathbb{R}\rightarrow [0,1]$. Let us see an example.

\begin{example}
Consider the fuzzy valued function
$$
F_{\textbf{C}_{\nu}^2}(x)=(1,2,3) \cdot x+(1,2,3)\cdot x^2,
\quad 0<x<2.
$$
By \eqref{equ1} and \eqref{eq3}, two parametric representations can be considered:
\begin{align}\nonumber
[F_{\textbf{C}_{\nu}^2}(x)]_{\alpha}&=\{(1+\alpha)x+(1+\alpha)x^2
+t\left((2-2\alpha)x+(2-2\alpha)x^2\right)~|~t\in[0,1]\},\\\label{equ2}
&=\{(1+\alpha+t(2-2\alpha))(x+x^2)~|~t\in[0,1]\},\\\label{equ3}
[F_{\textbf{C}_{\nu}^2}(x)]_{\alpha}&=\{f_{\textbf{c}(t, \alpha)}(x)
=(1+\alpha+t_1(2-2\alpha))x+(1+\alpha+t_2(2-2\alpha))x^2~|~
\textbf{c}(\textbf{t}, \alpha)\in [\textbf{C}_{\nu}^2]_{\alpha}\},
\end{align}
respectively, where
$[\textbf{C}_{\nu}^2]_{\alpha}=\{\textbf{c}(\textbf{t}, \alpha)~|~\textbf{c}(\textbf{t},
\alpha)=(1+\alpha+t_1(2-2\alpha),1+\alpha+t_2(2-2\alpha));~\textbf{t}=(t_1,t_2)\in [0,1]^2\}$.
By putting $t_1=0$ and $t_2=\frac{1}{2}$ in \eqref{equ3}, function $(1+\alpha)x+2x^2$ is obtained.
This function can be also obtained by setting $t=\frac{x^2}{2(x+x^2)}$ in representation \eqref{equ2}.
Thus, in the parametric representation \eqref{eq3} the parameter $t$ is independent from variable $x$,
which is in contrast with the parametric representation \eqref{equ1}.
\end{example}

\begin{remark}
If a fuzzy valued function does not contain the same fuzzy numbers (e.g., in coefficients),
then the parametric representation \eqref{eq3} coincides with that obtained from the use of CIA.
Otherwise, equivalence is not guaranteed. To see this, consider, e.g., the fuzzy valued function
\begin{equation}
\label{rem:ex:fvf}
F_{\textbf{C}_{\nu}^3}(x)=(2,3,4) \cdot \mathrm{Ln}(x)+(0,1,3)
\cdot e^x+(2,3,4)\cdot \sin(x), \quad 0<x<2\pi.
\end{equation}
Based on the parametric representations \eqref{equ1} and \eqref{eq3}, one has
\begin{align*}
[F_{\textbf{C}_{\nu}^3}(x)]_{\alpha}&=\left\lbrace
\begin{array}{ll}
(4-\alpha)\mathrm{Ln}(x)+\alpha e^x+(2+\alpha)\sin(x)&\\
\hspace{3cm} + t \left( (-2+2\alpha)\mathrm{Ln}(x)+(3-3\alpha)e^x+(2-2\alpha)\sin(x)\right),& 0<x\leq 1, \\
(2+\alpha)\mathrm{Ln}(x)+\alpha e^x+(2+\alpha)\sin(x)&\\
\hspace{3cm} + t \left((2-2\alpha)\mathrm{Ln}(x)+(3-3\alpha)e^x+(2-2\alpha)\sin(x)\right),& 1<x\leq \pi, \\
(2+\alpha)\mathrm{Ln}(x)+\alpha e^x+(4-\alpha)\sin(x)&\\
\hspace{3cm} + t\left((2-2\alpha)\mathrm{Ln}(x)+(3-3\alpha)e^x+(-2+2\alpha)\sin(x)\right),& \pi<x< 2\pi, \\
\end{array} \right.
\end{align*}
and
\begin{align*}
\begin{array}{ll}
[F_{\textbf{C}_{\nu}^3}(x)]_{\alpha}
&=\{(2+\alpha+t_1(2-2\alpha))\mathrm{Ln}(x)+(\alpha+t_2(3-3\alpha))e^x\\
&\hspace{6.5cm}+(2+\alpha+t_3(2-2\alpha))\sin(x), \quad t_1,t_2,t_3 \in [0,1]\}, \\
\end{array}
\end{align*}
respectively, while via SLCIA and CIA we obtain
\begin{align*}
[F_{\textbf{C}_{\nu}^3}(x)]_{\alpha}
&=\{(2+\alpha)\mathrm{Ln}(x)+\alpha e^x+(2+\alpha)\sin(x)\\
&\hspace{3cm} + t \left((2-2\alpha)\mathrm{Ln}(x)+(3-3\alpha)e^x+(2-2\alpha)\sin(x)\right), ~~~~t\in [0,1]\}
\end{align*}
and
\begin{align*}
[F_{\textbf{C}_{\nu}^3}(x)]_{\alpha}
&=\left\{\left(2+\alpha+t(2-2\alpha)\right)\left(\mathrm{Ln}(x)+\sin(x)\right)
+\left(\alpha+t'(3-3\alpha)\right)e^x,~ t,t'\in [0,1]\right\},
\end{align*}
respectively. Therefore, the fuzzy valued function \eqref{rem:ex:fvf}
has four distinct parametric representations.
\end{remark}

The following result establishes a relationship between the limit and continuity
of a fuzzy valued function and the limit and continuity of its $\alpha$-level set.

\begin{proposition}
\label{prop1} Let ${F_{\textbf{C}_{\nu}^k}: {\cal
S}\subseteq\mathbb{R} \rightarrow \mathcal{F}(\mathbb{R})}$ be a
fuzzy valued function and
$$
[F_{\textbf{C}_{\nu}^k}(x)]_{\alpha} = \left\{
f_{\textbf{c}(\textbf{t}, \alpha)}(x)| f_{\textbf{c}(\textbf{t},
\alpha)}: {\cal S}\subseteq\mathbb{R} \rightarrow \mathbb{R},
\textbf{c}(\textbf{t}, \alpha) \in [\textbf{C}_{\nu}^k]_{\alpha}
\right\}.
$$
If $\underset{x\rightarrow x_0}{\lim} f_{\textbf{c}(\textbf{t}, \alpha)}(x)$
exists for every $\textbf{c}(\textbf{t}, \alpha)\in[\textbf{C}_{\nu}^k]_{\alpha}$,
in other words, for every $\textbf{t}\in[0, 1]^k$, then
$\underset{x\rightarrow x_0}{\lim}F_{\textbf{C}_{\nu}^k}(x)$ exists and
\begin{equation*}
\left[\underset{x\rightarrow x_0}{\lim} F_{\textbf{C}_{\nu}^k}(x)\right]_{\alpha}
= \left[F_{\textbf{C}_{\nu}^k}(x_0)\right]_{\alpha}
= \left[\underset{\textbf{t}}{\min}\underset{x\rightarrow x_0}{\lim}
f_{\textbf{c}(\textbf{t}, \alpha)}(x), \underset{\textbf{t}}{\max}
\underset{x\rightarrow x_0}{\lim}f_{\textbf{c}(\textbf{t}, \alpha)}(x)\right].
\end{equation*}
Moreover, $F_{\textbf{C}_{\nu}^k}$ is continuous at $x_0$ if
$f_{\textbf{c}(\textbf{t}, \alpha)}$ is continuous at $x_0$
for every $\textbf{t}\in[0, 1]^k$ and $\alpha\in[0, 1]$.
\end{proposition}

\begin{proof}
Suppose $\lim_{x\rightarrow x_0} f_{\textbf{c}(\textbf{t}, \alpha)}(x)
=a(\textbf{t},\alpha)$ for every $\textbf{c}(\textbf{t}, \alpha)
\in [\textbf{C}_{\nu}^k]_{\alpha}$. Therefore, for every $\textbf{t}
\in[0,1]^k$ and $\alpha\in[0,1]$, it can be concluded that
\begin{equation}
\label{equ22}
\forall \epsilon>0, ~\exists \delta>0 : ~|x-x_0|<\delta
\Rightarrow |f_{\textbf{c}(\textbf{t},\alpha)}(x)-a(\textbf{t},\alpha)|
<\epsilon.
\end{equation}
On the other hand, because for every fixed $x\in {\cal S}$ and
$\alpha\in[0,1]$ the functions $f_{\textbf{c}(\textbf{t},
\alpha)}(x)$ and $a(\textbf{t},\alpha)$ are continuous in
$\textbf{t}$, there exist $\textbf{t}',\textbf{t}''\in[0,1]^k$ such
that
$$
\underset{\textbf{t}}{\min}|f_{\textbf{c}(\textbf{t},\alpha)}(x)-a(\textbf{t},\alpha)|
=|f_{\textbf{c}(\textbf{t}',\alpha)}(x)-a(\textbf{t}',\alpha)|,
$$
$$
\underset{\textbf{t}}{\max}|f_{\textbf{c}(\textbf{t},\alpha)}(x)
-a(\textbf{t},\alpha)|=|f_{\textbf{c}(\textbf{t}'',\alpha)}(x)
-a(\textbf{t}'',\alpha)|.
$$
Then, from \eqref{equ22}, $\forall \epsilon>0, \exists ~\delta_1, \delta_2>0$ such that
$$
|x-x_0|<\delta_1 \Rightarrow |f_{\textbf{c}(\textbf{t}',\alpha)}(x)
-a(\textbf{t}',\alpha)|<\epsilon,
$$
$$
|x-x_0|<\delta_2 \Rightarrow |f_{\textbf{c}(\textbf{t}'',\alpha)}(x)-a(\textbf{t}'',\alpha)|<\epsilon.
$$
Furthermore, $f_{\textbf{c}(\textbf{t},\alpha)}$ satisfies Theorem~\ref{th22}
(the stacking theorem). Then, $a(\textbf{t},\alpha)$ fulfills the same property,
that is, there is a fuzzy number $A$ such that
$[A]_\alpha=\{a(\textbf{t},\alpha)| \textbf{t}\in[0,1]^k\}$.
Next, by choosing $\overline{\delta}=\min\{\delta_1,\delta_2\}$,
$$
D(F_{\textbf{C}_{\nu}^k}(x), A)
=\max\left\{ \underset{\textbf{t}}{\min}|f_{\textbf{c}(\textbf{t},\alpha)}(x)-a(\textbf{t},\alpha)|,
\underset{\textbf{t}}{\max}|f_{\textbf{c}(\textbf{t},\alpha)}(x)-a(\textbf{t},\alpha)| \right\}<\epsilon,
$$
whenever $|x-x_0|<\overline{\delta}$. Hence, by Definition~\ref{def6},
$\underset{x\rightarrow x_0}{\lim} F_{\textbf{C}_{\nu}^k}(x)=A$. Moreover,
\begin{align*}
\left[\underset{x\rightarrow x_0}{\lim} F_{\textbf{C}_{\nu}^k}(x)\right]_{\alpha} = [A]_{\alpha}
&= \left[\underset{\textbf{t}}{\min} ~a(\textbf{t}, \alpha),
\underset{\textbf{t}}{\max}~a(\textbf{t},\alpha)\right]\\
&=\left[\underset{\textbf{t}}{\min}\underset{x\rightarrow x_0}{\lim}
f_{\textbf{c}(\textbf{t}, \alpha)}(x), \underset{\textbf{t}}{\max}
\underset{x\rightarrow x_0}{\lim}f_{\textbf{c}(\textbf{t}, \alpha)}(x)\right].
\end{align*}
The continuity of $F_{\textbf{C}_{\nu}^k}$ at $x_0$ is proved similarly,
by using the continuity of $f_{\textbf{c}(\textbf{t},\alpha)}$ at $x_0$.
\end{proof}

\begin{proposition}
\label{prop4} If ${F_{\textbf{C}_{\nu}^k}: {\cal
S}\subseteq\mathbb{R}\rightarrow \mathcal{F}(\mathbb{R})}$ is
continuous at $x_0\in {\cal S}$ and
$[F_{\textbf{C}_{\nu}^k}(x)]_{\alpha} = \{ f_{\textbf{c}(\textbf{t},
\alpha)}(x)| $\\$f_{\textbf{c}(\textbf{t}, \alpha)}: {\cal
S}\subseteq\mathbb{R} \rightarrow\mathbb{R}, \textbf{c}(\textbf{t},
\alpha) \in [\textbf{C}_{\nu}^k]_{\alpha} \}$, then
$f_{\textbf{c}(\textbf{t}, \alpha)}$ is a continuous function at
$x_0$ for each $\textbf{t}\in[0,1]^k$ and $\alpha\in[0,1]$.
\end{proposition}

\begin{proof}
Since $F_{\textbf{C}_{\nu}^k}$ is continuous, from
Definition~\ref{def1}, Definition~\ref{def7} and
Proposition~\ref{prop3} it can be concluded that
\begin{align*}
\underset{x\rightarrow x_0}{\lim}F_{\textbf{C}_{\nu}^k}(x)
&=F_{\textbf{C}_{\nu}^k}(x_0) \Leftrightarrow
\underset{x\rightarrow x_0}{\lim}(F_{\textbf{C}_{\nu}^k}(x)\ominus_p 
F_{\textbf{C}_{\nu}^k}(x_0))=0\\
&\Leftrightarrow \{\underset{x\rightarrow
x_0}{\lim}(f_{\textbf{c}(\textbf{t}, \alpha)}(x)
-f_{\textbf{c}(\textbf{t}, \alpha)}(x_0))| f_{\textbf{c}(\textbf{t},
\alpha)}: {\cal S}\subseteq\mathbb{R} \rightarrow \mathbb{R},
\textbf{c}(\textbf{t}, \alpha) \in [\textbf{C}_{\nu}^k]_{\alpha}
\}=0.
\end{align*}
Therefore, for each $\textbf{t}\in[0,1]^k$ and $\alpha\in[0,1]$,
\begin{equation*}
\underset{x\rightarrow x_0}{\lim}(f_{\textbf{c}(\textbf{t}, \alpha)}(x)
-f_{\textbf{c}(\textbf{t}, \alpha)}(x_0))=0\Rightarrow
\underset{x\rightarrow x_0}{\lim}f_{\textbf{c}(\textbf{t}, \alpha)}(x)
=f_{\textbf{c}(\textbf{t}, \alpha)}(x_0).
\end{equation*}
The proof is complete.
\end{proof}

\begin{remark}
\label{remk1} Similarly to Proposition~\ref{prop4}, it can be shown
that if $F: {\cal S}\subseteq\mathbb{R} \rightarrow
\mathcal{F}(\mathbb{R})$ is a continuous fuzzy valued function with
$[F(x)]_{\alpha}
=\left\{f^-_{\alpha}(x)+t(f^+_{\alpha}(x)-f^-_{\alpha}(x))|
t\in[0,1] \right\}$, then the real valued functions
$f^-_{\alpha}(x)$ and $f^+_{\alpha}(x)$ are continuous.
\end{remark}


\subsection{Differentiation of fuzzy valued functions}

Based on the notion of $p$-difference, we start this section
with the definition of $p$-differentiability
of a fuzzy valued function.

\begin{definition}
\label{defi1}
Let $x_0 \in \, ]a, b[$ and $h$ be such that $x_0+h\in]a,b[$. Then the
$p$-derivative of the fuzzy valued function
$F: ]a,b[ \rightarrow \mathcal{F}(\mathbb{R})$ at $x_0$ is defined as
\begin{equation*}
F'_p(x_0)=\lim_{h\rightarrow 0}
\frac{1}{h}\left[F(x_0+h)\ominus_p F(x_0)\right].
\end{equation*}
\end{definition}

If $F'_p(x_0)\in \mathcal{F}(\mathbb{R})$ exists, then we say that $F$
is parametric differentiable ($p$-differentiable, for short) at $x_0$.

\begin{proposition}
\label{prop2}
Let $F_{\textbf{C}_{\nu}^k} : \, ]a,b[ \rightarrow \mathcal{F}(\mathbb{R})$
be defined in terms of its $\alpha$-level set
$$
[F_{\textbf{C}_{\nu}^k}]_{\alpha}
=\left\{f_{\textbf{c}(\textbf{t},\alpha)}(x)|
f_{\textbf{c}(\textbf{t},\alpha)}: ]a,b[\rightarrow \mathbb{R},
\textbf{c}(\textbf{t},\alpha)\in[\textbf{C}_{\nu}^k]_\alpha\right\}.
$$
If $f_{\textbf{c}(\textbf{t},\alpha)}(x)$ is differentiable at
$x_0 \in \, ]a,b[$ and for all $\alpha \in [0, 1]$
$f_{\textbf{c}(\textbf{t}, \alpha)}(x_0+h)-f_{\textbf{c}(\textbf{t}, \alpha)}(x_0)$
satisfy the stacking theorem, then $F_{\textbf{C}_{\nu}^k}$ is $p$-differentiable at $x_0$
and there exists $(F_{\textbf{C}_{\nu}^k})_p'(x_0)\in \mathcal{F}(\mathbb{R})$ such that
\begin{equation*}
[(F_{\textbf{C}_{\nu}^k})_p'(x_0)]_{\alpha}
=\left\{f'_{\textbf{c}(\textbf{t},\alpha)}(x_0)
\Big| f_{\textbf{c}(\textbf{t},\alpha)}: ]a,b[\rightarrow\mathbb{R},
\textbf{c}(\textbf{t},\alpha)\in[\textbf{C}_{\nu}^k]_\alpha\right\}.
\end{equation*}
Moreover,
\begin{equation}
\label{eq11}
[(F_{\textbf{C}_{\nu}^k})_p'(x_0)]_{\alpha}=\left[\underset{\textbf{t}}{\min}
f'_{\textbf{c}(\textbf{t},\alpha)}(x_0), \underset{\textbf{t}}{\max}
f'_{\textbf{c}(\textbf{t},\alpha)}(x_0) \right].
\end{equation}
\end{proposition}

\begin{proof}
Using Definition~\ref{def3}, we can write that
\begin{equation*}
\left[F_{\textbf{C}_{\nu}^k}(x_0+h) \ominus_p
F_{\textbf{C}_{\nu}^k}(x_0)\right]_{\alpha}
= \left\{ f_{\textbf{c}(\textbf{t}, \alpha)}(x_0+h)
- f_{\textbf{c}(\textbf{t}, \alpha)}(x_0)\big|
~f_{\textbf{c}(\textbf{t}, \alpha)}: ]a,b[ \rightarrow \mathbb{R},
~\textbf{c}(\textbf{t}, \alpha) \in [\textbf{C}_{\nu}^k]_{\alpha} \right\}.
\end{equation*}
Because $f_{\textbf{c}(\textbf{t}, \alpha)}$ is differentiable at $x_0$
for every $\textbf{t}\in[0,1]^k$ and $\alpha\in[0,1]$,
$\underset{h\rightarrow 0}{\lim}\frac{f_{\textbf{c}(\textbf{t}, \alpha)}(x_0+h)
-f_{\textbf{c}(\textbf{t}, \alpha)}(x_0)}{h}$ exists and
\begin{equation}
\label{eq111}
\begin{split}
&\left[\underset{h\rightarrow 0}{\lim}\frac{1}{h}\left[F_{\textbf{C}_{\nu}^k}(x_0+h)
\ominus_p F_{\textbf{C}_{\nu}^k}(x_0)\right]\right]_{\alpha}\\
&\quad =\left\{ \underset{h\rightarrow 0}{\lim}\frac{f_{\textbf{c}(\textbf{t}, \alpha)}(x_0+h)
-f_{\textbf{c}(\textbf{t}, \alpha)}(x_0)}{h}\Big| f_{\textbf{c}(\textbf{t},\alpha)}
: ]a,b[\rightarrow\mathbb{R}, \textbf{c}(\textbf{t}, \alpha)
\in [\textbf{C}_{\nu}^k]_{\alpha} \right\}\\
&\quad =\left\{f'_{\textbf{c}(\textbf{t},\alpha)}(x_0)\Big|
f_{\textbf{c}(\textbf{t},\alpha)}: ]a,b[\rightarrow\mathbb{R},
\textbf{c}(\textbf{t},\alpha)\in[\textbf{C}_{\nu}^k]_\alpha\right\}.
\end{split}
\end{equation}
By assumption, $f_{\textbf{c}(\textbf{t}, \alpha)}(x_0+h)
-f_{\textbf{c}(\textbf{t}, \alpha)}(x_0)$ is an $\alpha$-level set of a fuzzy number.
Thus, there exists $(F_{\textbf{C}^k_\nu})_p'(x_0)\in\mathcal{F}(\mathbb{R})$ such that
\begin{equation}
\label{eq1111}
\left[\underset{h\rightarrow 0}{\lim}\frac{1}{h}\left[F_{\textbf{C}_{\nu}^k}(x_0+h)
\ominus_p F_{\textbf{C}_{\nu}^k}(x_0)\right]\right]_{\alpha}
= [(F_{\textbf{C}^k_\nu})_p'(x_0)]_{\alpha}.
\end{equation}
Equation \eqref{eq11} follows immediately from \eqref{eq111}, \eqref{eq1111},
and the continuity of $f'_{\textbf{c}(\textbf{t},\alpha)}(x_0)$
in $\textbf{t}$ for all $\alpha\in[0,1]$.
\end{proof}

\begin{example}
Consider the fuzzy valued function
$F_{\textbf{C}_{\nu}^1}(x)=(-1,1,2)\cdot e^{-x}$ and the
$\alpha$-level set $ [(-1,1,2)]_{\alpha}=[-1+2\alpha,~2-\alpha]$. By
\eqref{eq3} the parametric representation of
$F_{\textbf{C}_{\nu}^1}(x)$ is
$$
[F_{\textbf{C}_{\nu}^1}(x)]_{\alpha}
=\left\{(-1+2\alpha+t(3-3\alpha))e^{-x} | t\in[0,1] \right\}.
$$
It is possible to calculate its $p$-derivative from Proposition~\ref{prop2} as
$$
[(F_{\textbf{C}^1_\nu})_p'(x)]_{\alpha}=\{(1-2\alpha+t(3\alpha-3))e^{-x}
| t\in[0,1] \}.
$$
Then by \eqref{eq11}
$$[F_{\textbf{C}_{\nu}^1}(x)]_{\alpha}=[(-2+\alpha)e^{-x},~(1-2\alpha)e^{-x}],$$
that is, $(F_{\textbf{C}^1_\nu})_p'(x)=(-2,-1,1) \cdot e^{-x}$.
\end{example}

The following example shows that the converse
of Proposition~\ref{prop2} is not true.

\begin{example}
Let $F_{\textbf{C}_{\nu}^1}(x)=A \cdot g(x)$
be a fuzzy valued function with $A=(-3,-1,1,3)$ and
\begin{equation*}
g(x)=\left\{
\begin{array}{ll}
x,&~ x\geq0\\
-x,&~ x<0.\\
\end{array}
\right.
\end{equation*}
From Definition~\ref{defi1},
\begin{align*}
[(F_{\textbf{C}_{\nu}^1})'_p(0)]_{\alpha}
&=\underset{h\rightarrow 0}{\lim}\frac{1}{h}[F_{\textbf{C}_{\nu}^1}(h)
\ominus_p F_{\textbf{C}_{\nu}^1}(0)]_{\alpha}\\
&=\left\lbrace \underset{h\rightarrow 0}{\lim}\frac{f_{\textbf{c}(t,\alpha)}(h)
- f_{\textbf{c}(t,\alpha)}(0)}{h}~\Bigg{|}~ f_{\textbf{c}(t,\alpha)}(x)
=\left\lbrace
\begin{array}{ll}
(-3+2\alpha+t(6-4\alpha))(x),~~~~x\geq 0\\
(-3+2\alpha+t(6-4\alpha))(-x),~~x<0
\end{array}\right.;
t \in [0,1] \right\rbrace\\
&=\left\lbrace f_{\textbf{c}(t,\alpha)}(x)~\Bigg{|}
~ f_{\textbf{c}(t,\alpha)}(x)
=\left\lbrace
\begin{array}{ll}
-3+2\alpha+t(6-4\alpha),~~x \geq 0\\
3-2\alpha+t(-6+4\alpha),~~x<0
\end{array}\right.; t \in [0,1]\right\rbrace
=[-3+2\alpha,3-2\alpha].
\end{align*}
Then, $F_{\textbf{C}_{\nu}^1}(x)$ is $p$-differentiable at $x=0$
and $(F_{\textbf{C}_{\nu}^1})'_p(0)=A$ but
$$
f_{\textbf{c}(t,\alpha)}(x)
=\left\lbrace
\begin{array}{ll}
(-3+2\alpha+t(6-4\alpha))x,&~~~x \geq 0\\
(-3+2\alpha+t(6-4\alpha))(-x),&~~~x<0
\end{array}\right.
$$
is not differentiable at $x=0$ for every $t,\alpha \in [0,1]$.
\end{example}

\begin{theorem}
\label{th3}
Let $F : \, ]a,b[ \rightarrow \mathcal{F}(\mathbb{R})$ be defined
in terms of its $\alpha$-level set as
$[F(x)]_{\alpha}=\{ f_{(t,\alpha)}(x)=f^-_{\alpha}(x)
+t(f^+_{\alpha}(x)-f^-_{\alpha}(x)) |~ t\in[0,1]\}$.
Suppose that function $f_{(t,\alpha)}(x)$ is a real valued function,
differentiable w.r.t. $x$ and $t$, uniformly w.r.t. $\alpha\in[0,1]$.
Then, function $F$ is $p$-differentiable at a fixed $x_0 \in \, ]a, b[$
if and only if one of the following conditions is satisfied:
$$
\left\{
\begin{array}{ll}
\left(\Delta_{\alpha}f_{(0,\alpha)}\right)'(x_0)\geq 0,\\
\left(\Delta_{\alpha}f_{(1,\alpha)}\right)'(x_0)\leq 0,\\
\left(\Delta_tf_{(t,1)}\right)'(x_0)\geq 0,
\end{array}
\right.
\quad \text{ or } \quad
\left\{
\begin{array}{ll}
\left(\Delta_{\alpha}f_{(0,\alpha)}\right)'(x_0)\leq0,\\
\left(\Delta_{\alpha}f_{(1,\alpha)}\right)'(x_0)\geq0,\\
\left(\Delta_tf_{(t,1)}\right)'(x_0)\leq0,
\end{array}
\right.
$$
for all $\alpha\in[0,1]$. Moreover, $p$-differentiability
and generalized Hukuhara differentiability coincide.
\end{theorem}

\begin{proof}
The first part of this theorem can be proved by Lemma~\ref{lem1}.
For the second part, using Proposition~\ref{prop2}, we have
\begin{align*}
&[F_p'(x_0)]_{\alpha}=\left\{\lim_{h\rightarrow 0}
\frac{f^-_\alpha(x_0+h)+t(f^+_\alpha(x_0+h)-f^-_\alpha(x_0+h))
-f^-_\alpha(x_0)-t(f^+_\alpha(x_0)-f^-_\alpha(x_0))}{h}\Big| t\in[0,1]\right\}\\
&\hspace{1cm}=\left\{ \lim_{h\rightarrow 0}
\frac{f^-_\alpha(x_0+h)-f^-_\alpha(x_0)}{h}+t\lim_{h\rightarrow 0}
\frac{f^+_\alpha(x_0+h)-f^+_\alpha(x_0)}{h}-t\lim_{h\rightarrow 0}
\frac{f^-_\alpha(x_0+h)-f^-_\alpha(x_0)}{h}\Big| t\in[0,1]\right\}\\
&\hspace{1cm}=\left\{ (f^-_{\alpha})'(x_0)+t((f^+_{\alpha})'(x_0)
-(f^-_{\alpha})'(x_0))| t\in[0,1] \right\}.
\end{align*}
If the first condition hold, then $(f^-_{\alpha})'(x_0)
\leq(f^+_{\alpha})'(x_0)$ for all $\alpha \in [0, 1]$.
Thus, it follows from the non-decreasing representation that
\begin{equation*}
[F_p'(x_0)]_{\alpha}=[(f^-_{\alpha})'(x_0), (f^+_{\alpha})'(x_0)],
\end{equation*}
which is coincident with the (i)-gH-differentiability concept.
If the second condition hold, then
$(f^-_{\alpha})'(x_0)\geq(f^+_{\alpha})'(x_0)$ for all $\alpha \in [0, 1]$.
Thus, by the non-increasing representation,
\begin{equation*}
[F_p'(x_0)]_{\alpha}=[(f^+_{\alpha})'(x_0), (f^-_{\alpha})'(x_0)],
\end{equation*}
which is coincident with the (ii)-gH-differentiability concept.
\end{proof}

\begin{remark}
\label{remk2}
Due to Definition~\ref{defi1} of $p$-derivative and the two parametric
representations \eqref{equ1} and \eqref{eq3} for fuzzy valued functions,
it cannot be expected, in general, that the derivative of a fuzzy valued
function coincides on the two representations.
In fact, the sign of the independent variable $x$ is not considered in the
$p$-derivative, based on representation \eqref{eq3}, while the
$p$-derivative on representation \eqref{equ1} depends on the sign of $x$.
For example, consider the fuzzy valued function
$F_{\textbf{C}_{\nu}^2}(x)=(-1,2,3)x+(2,4,5)e^x$,
$x\in [-3,6]$. Its corresponding parametric representations are
\begin{equation*}
[F_{\textbf{C}_{\nu}^2}(x)]_{\alpha}=
\left\lbrace\begin{array}{ll}
\{(3-\alpha)x+(2+2\alpha)e^x+t((-4+2\alpha)x+(3+\alpha)e^x)~|~t\in[0,1]\},~-3\leq x \leq 0\\
\{(-1+3\alpha)x+(2+2\alpha)e^x+t((4-2\alpha)x+(3-\alpha)e^x)~|~t\in[0,1]\},~~~ 0\leq x \leq 6
\end{array}\right.
\end{equation*}
and
$$
[F_{\textbf{C}_{\nu}^2}(x)]_{\alpha}=\{(-1+3\alpha+t_1(4-4\alpha))x
+(2+2\alpha+t_2(3-3\alpha))e^x|t_1,t_2 \in [0,1]\}.
$$
Their $p$-derivatives are given by
\begin{align*}
[F'_{\textbf{C}_{\nu}^2}(x)]_{\alpha}
&=
\left\lbrace\begin{array}{lc}
\{(3-\alpha)+(2+2\alpha)e^x+t((-4+2\alpha)+(3+\alpha)e^x)~|~t\in[0,1]\},&-3\leq x \leq 0\\
\{(-1+3\alpha)+(2+2\alpha)e^x+t((4-2\alpha)+(3-\alpha)e^x)~|~t\in[0,1]\},&~~0\leq x \leq 6
\end{array}\right.\\
&=\left\lbrace\begin{array}{ll}
~[3-\alpha+(2+2\alpha)e^x,-1+3\alpha+(5-\alpha)e^x],&-3\leq x \leq 0\\
~[-1+3\alpha+(2+2\alpha)e^x,3-\alpha+(5-\alpha)e^x],&~~0\leq x \leq 6
\end{array}\right.
\end{align*}
and
\begin{align*}
[F'_{\textbf{C}_{\nu}^2}(x)]_{\alpha}&=\{(-1+3\alpha+t_1(4-4\alpha))
+(2+2\alpha+t_2(3-3\alpha))e^x|t_1,t_2 \in [0,1]\}\\
&=[-1+3\alpha+(2+2\alpha)e^x,3-\alpha+(5-\alpha)e^x],
\end{align*}
respectively, which are clearly not the same.
\end{remark}

According to Theorem~\ref{th3}, when $f_{(t,\alpha)}(x)$ is differentiable,
two cases can be considered for the definition of $p$-differentiability,
corresponding to the non-decreasing and non-increasing parametric
representations \eqref{eq1} and \eqref{eq2}.

\begin{definition}
\label{def777}
Let $F : \, ]a,b[ \rightarrow \mathcal{F}(\mathbb{R})$ be a fuzzy valued function and
$$
[F(x)]_{\alpha}=\left\{ f_{(t,\alpha)}(x)
=f^-_{\alpha}(x)+t(f^+_{\alpha}(x)-f^-_{\alpha}(x)) | t\in[0,1]\right\}
$$
with $f_{(t,\alpha)}(x)$ differentiable at $x_0\in]a,b[$. Then,
\begin{itemize}
\item  $F$ is called $i$-$p$-differentiable at $x_0$ if
\begin{equation}
\label{eq33}
[F_p'(x_0)]_{\alpha}=\{(f^-_\alpha)'(x_0)
+t((f^+_\alpha)'(x_0)-(f^-_\alpha)'(x_0)) |~ t\in[0,1] \},
\end{equation}
\item  $F$ is called $d$-$p$-differentiable at $x_0$ if
\begin{equation}
\label{eq333}
[F_p'(x_0)]_{\alpha}=\{ (f^+_\alpha)'(x_0)+t((f^-_\alpha)'(x_0)
-(f^+_\alpha)'(x_0)) |~ t\in[0,1] \}.
\end{equation}
\end{itemize}
\end{definition}

The concept of switching point can be extended as follows.

\begin{definition}
\label{def77}
A point $x_0 \in \, ]a,b[$ is said to be a switching point for the differentiability of $F$,
if in any neighbourhood $N$ of $x_0$ there exist points $x_1<x_0<x_2$ such that
\begin{description}
\item[type-I:] at $x_1$ \eqref{eq33} holds while \eqref{eq333} does not hold and, at $x_2$,
\eqref{eq333} holds and \eqref{eq33} does not hold;

\item[type-II:] at $x_1$ \eqref{eq333} holds while \eqref{eq33} does not hold and,
at $x_2$, \eqref{eq33} holds and \eqref{eq333} does not hold.
\end{description}
\end{definition}

Using Definition~\ref{def77} and Theorem~\ref{th3},
it is easy to find switching points.

\begin{example}
\label{exam1}
Let us consider the fuzzy valued function
$F : \, ]-10,10[ \rightarrow \mathcal{F}(\mathbb{R})$ defined by
\begin{equation*}
F(x)=(2,4,5,8) \cdot \left(\cos(x)-\frac{x^2}{32}\right).
\end{equation*}
Its $\alpha$-level set is
\begin{equation*}
[F(x)]_{\alpha}
=
\begin{cases}
[(8-3\alpha)(\cos(x)-\frac{x^2}{32}),
(2+2\alpha)(\cos(x)-\frac{x^2}{32})],
& -10<x<-1.5004\\
[(2+2\alpha)(\cos(x)-\frac{x^2}{32}), (8-3\alpha)(\cos(x)-\frac{x^2}{32})],
& -1.5004\leq x<1.5004\\
[(8-3\alpha)(\cos(x)-\frac{x^2}{32}), (2+2\alpha)(\cos(x)-\frac{x^2}{32})],
& 1.5004\leq x<10,
\end{cases}
\end{equation*}
which is illustrated in Figure~\ref{fig1:sub1}.
We have
\begin{align*}
\left(\Delta_{\alpha}f_{(0,\alpha)}\right)'(x)
&=
\left\{
\begin{array}{ll}
3\sin(x)+\frac{3x}{16}, ~~~~~~~~~~ -10<x<-1.5004\\
-2\sin(x)-\frac{x}{8}, ~~ -1.5004\leq x<1.5004\\
3\sin(x)+\frac{3x}{16}, ~~~~~~~~ 1.5004\leq x<10,
\end{array}
\right.\\
\left(\Delta_{\alpha}f_{(1,\alpha)}\right)'(x)
&=
\left\{
\begin{array}{ll}
-2\sin(x)-\frac{x}{8}, ~~~~~~~~ -10<x<-1.5004\\
3\sin(x)+\frac{3x}{16}, ~~~~ -1.5004\leq x<1.5004\\
-2\sin(x)-\frac{x}{8}, ~~~~~ 1.5004\leq x<10,
\end{array}
\right.\\
\left(\Delta_tf_{(t,1)}\right)'(x)
&= \left\{
\begin{array}{ll}
\sin(x)+\frac{x}{16}, ~~~~~~~~~~~ -10<x<-1.5004\\
-\sin(x)-\frac{x}{16}, ~~ -1.5004\leq x<1.5004\\
\sin(x)+\frac{x}{16}, ~~~~~~~~~ 1.5004\leq x<10.
\end{array}
\right.
\end{align*}
The derivatives of $\Delta_{\alpha}f_{(0,\alpha)},
\Delta_{\alpha}f_{(1,\alpha)}$ and $\Delta_tf_{(t,1)}$ are given
in Figure~\ref{fig1:sub2} and, from Definition~\ref{def77} and Theorem~\ref{th3},
it can be deduced that the function is $p$-differentiable and the points
$x_1=-1.5004$, $x_2=1.5004$, $x_3=-5.9052$ and $x_4=5.9052$ are switching points
of type-II. Additionally, the points $x_5=0$, $x_6=-3.3527$ and $x_7=3.3527$
are switching points of type-I.
\begin{figure}
\centering
{\includegraphics[height=3.9cm,width=0.5\linewidth]{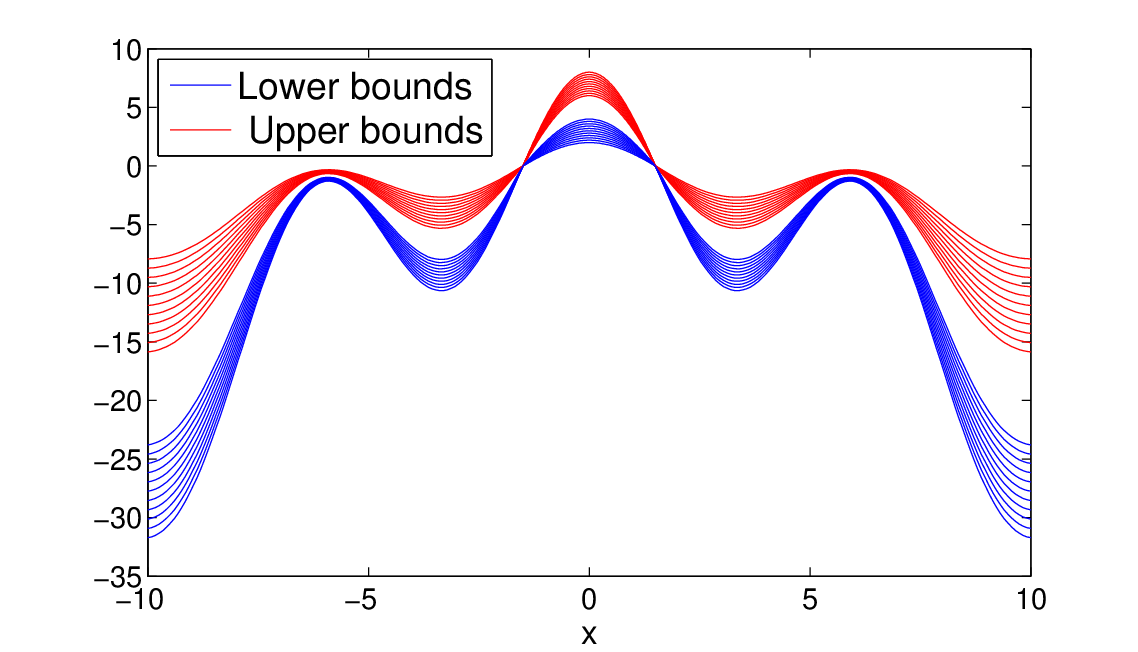}\label{fig1:sub1}}
{\includegraphics[height=3.9cm,width=0.5\linewidth]{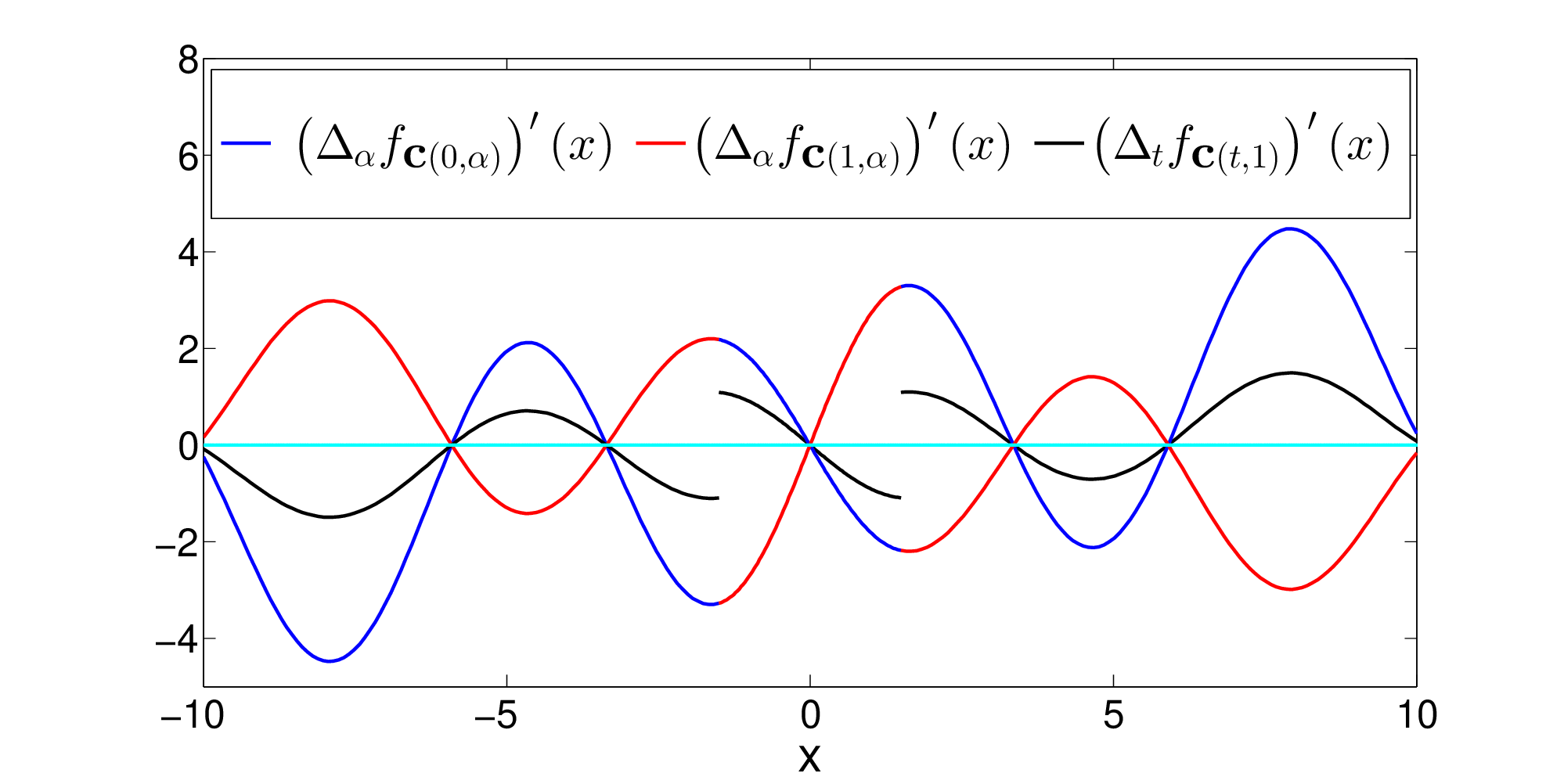}\label{fig1:sub2}}
\caption{Function of Example~\ref{exam1}, (a) $\alpha$-level set,
(b) derivatives of the $\alpha$-level set.} \label{fig1}
\end{figure}
\end{example}

Based on the notion of $gp$-difference introduced in Definition~\ref{def33},
the following $gp$-differentiability concept is proposed,
which further extends the notion of $p$-differentiability.

\begin{definition}
\label{def8}
Let $x_0 \in \, ]a, b[$ and $h$ be such that $x_0+h \in \, ]a,b[$.
Then the $gp$-derivative of the fuzzy valued function
$F : \, ]a,b[ \rightarrow \mathcal{F}(\mathbb{R})$ at $x_0$ is defined by
$$
F'_{gp}(x_0)=\lim_{h\rightarrow 0}\frac{1}{h}\left[F(x_0+h)
\ominus_{gp} F(x_0)\right].
$$
\end{definition}

If $F_{gp}(x_0)\in \mathcal{F}(\mathbb{R})$ exists, then $F$
is said to be generalized parametric differentiable
($gp$-differentiable, for short) at $x_0$.
In the following theorem, a characterization and a practical
formula for the $gp$-derivative is given.

\begin{theorem}
Let $F_{\textbf{C}_{\nu}^k} : \, ]a,b[ \rightarrow \mathcal{F}(\mathbb{R})$ be such that
$[F_{\textbf{C}_{\nu}^k}(x)]_{\alpha} =\{f_{\textbf{c}(\textbf{t}, \alpha)} (x)
| f_{\textbf{c}(\textbf{t}, \alpha)}: [a,b] \rightarrow \mathbb{R},
~\textbf{c}(\textbf{t}, \alpha)\in [\textbf{C}_{\nu}^k]_\alpha \}$.
If $f_{\textbf{c}(\textbf{t},\alpha)}$ is a differentiable real valued function w.r.t. $x$,
uniformly w.r.t. $\alpha\in[0,1]$, then $F_{\textbf{C}_{\nu}^k}(x)$
is $gp$-differentiable at $x_0 \in \, ]a,b[$ and
\begin{equation*}
[(F_{\textbf{C}_{\nu}^k})'_{gp}(x_0)]_{\alpha}
=\left[\inf_{\beta\geq\alpha}\min_{\textbf{t}}\left(
f'_{\textbf{c}(\textbf{t},\beta)}(x_0)\right),
\sup_{\beta\geq\alpha}\max_{\textbf{t}}\left(
f'_{\textbf{c}(\textbf{t},\beta)}(x_0)\right) \right].
\end{equation*}
\end{theorem}

\begin{proof}
From Definitions~\ref{def33} and \ref{def8},
\begin{align*}
[(F_{\textbf{C}_{\nu}^k})'_{gp}(x_0)]_{\alpha}
= \lim_{h\rightarrow 0}\frac{1}{h}\Big[\inf_{\beta\geq\alpha}\min_{\textbf{t}}(
f_{\textbf{c}(\textbf{t},\alpha)}(x_0+h)
&-f_{\textbf{c}(\textbf{t},\alpha)}(x_0)),\\
& \sup_{\beta\geq\alpha}\max_{\textbf{t}}( f_{\textbf{c}(\textbf{t},\alpha)}(x_0+h)
-f_{\textbf{c}(\textbf{t},\alpha)}(x_0)) \Big].
\end{align*}
Since $f_{\textbf{c}(\textbf{t},\alpha)}$ is differentiable, then
\begin{equation*}
[(F_{\textbf{C}_{\nu}^k})'_{gp}(x_0)]_{\alpha}
=\left[\inf_{\beta\geq\alpha}\min_{\textbf{t}}\left(
f'_{\textbf{c}(\textbf{t},\beta)}(x_0)\right),
\sup_{\beta\geq\alpha}\max_{\textbf{t}}\left(
f'_{\textbf{c}(\textbf{t},\beta)}(x_0)\right) \right]
\end{equation*}
for any $\alpha\in[0,1]$. The rest of the proof is similar to the proof
of Theorem~34 of \cite{Bede2013}.
\end{proof}

Similarly, if $F : \, ]a,b[ \rightarrow \mathcal{F}(\mathbb{R})$
is a fuzzy valued function with $\alpha$-level set
$$
[F]_{\alpha}=\left\{f_{(t,\alpha)}(x)=f^-_{\alpha}(x)
+t\left(f^+_{\alpha}(x)-f^-_{\alpha}(x)\right)| ~t\in[0,1]\right\},
$$
then
\begin{equation*}
[F'_{gp}(x_0)]_{\alpha}=\left[\inf_{\beta\geq\alpha}
\min_t\left(f'_{(t,\beta)}(x_0)\right), \sup_{\beta\geq\alpha}
\max_t\left( f'_{(t,\beta)}(x_0)\right) \right].
\end{equation*}

\begin{example}
\label{exam2}
Consider function $F: [0,1]\rightarrow \mathcal{F}(\mathbb{R})$
having the $\alpha$-level set
\begin{equation*}
[F(x)]_{\alpha}=\left\{f_{(t,\alpha)}(x)| f_{(t,\alpha)}(x)
=\alpha x^2+t(x^2+1-\alpha); t\in[0,1] \right\}.
\end{equation*}
The derivatives are
\begin{equation*}
\left( \Delta_{\alpha}f_{(0,\alpha)}\right)'(x)
=\left( \Delta_{\alpha}f_{(1,\alpha)}\right)'(x)
= \left( \Delta_tf_{(t,1)}\right)'(x)=2x.
\end{equation*}
From Theorem~\ref{th3} and Definition~\ref{def777},
the function is not $p$-differentiable but it is $gp$-differentiable:
\begin{align*}
[F'_{gp}(x)]_{\alpha}&=\left[\inf_{\beta\geq\alpha}\min_{t}(2\beta x+2tx),
\sup_{\beta\geq\alpha}\max_{t}(2\beta x+2tx)\right]\\
&=
\begin{cases}
[\underset{\beta\geq\alpha}\inf(2\beta x+2x),
\underset{\beta\geq\alpha}\sup(2\beta x)], & -1\leq x\leq 0\\
[\underset{\beta\geq\alpha}\inf(2\beta x),
\underset{\beta\geq\alpha}\sup(2\beta x+2x)], & 0\leq x\leq 1
\end{cases}\\
&=
\begin{cases}
[4x, 2\alpha x], & -1\leq x\leq 0\\
[2\alpha x, 4x], & 0\leq x\leq 1.
\end{cases}
\end{align*}
The $\alpha$-level set of $F'_{gp}(x)$ is shown in Figure~\ref{fig2}.
\begin{figure}
\centering{\includegraphics[height=5cm,width=10cm]{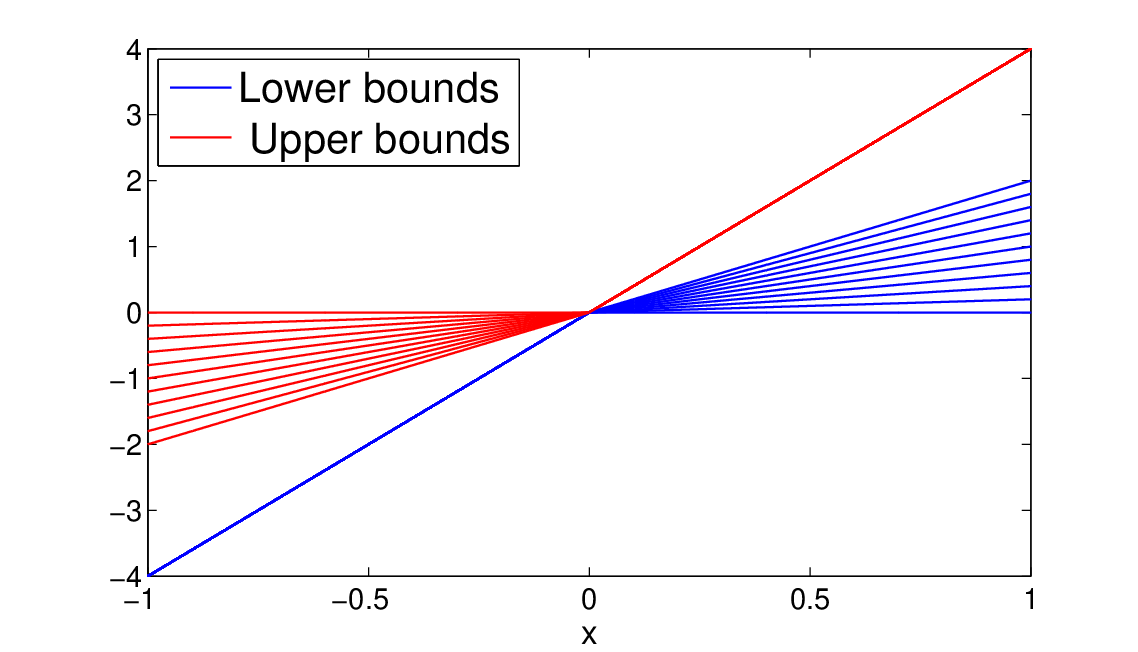}
\caption{\footnotesize{The $\alpha$-level set of the $gp$-derivative
of function in Example~\ref{exam2}.}}\label{fig2}}
\end{figure}
\end{example}

\begin{example}
\label{exam3}
Let us consider the function $F: [0,1]\rightarrow \mathcal{F}(\mathbb{R})$ defined by
\begin{equation*}
[F(x)]_{\alpha}=\left\{f_{(t,\alpha)}(x)| f_{(t,\alpha)}(x)
=x e^{-x}+\alpha^2(e^{-x^2}+x-xe^{-x})+t(1-\alpha^2)(2e^{-x^2}+e^x-xe^{-x}); t\in[0,1] \right\}.
\end{equation*}
We have
\begin{align*}
\left( \Delta_{\alpha}f_{(0,\alpha)}\right)'(x)
&=2\alpha(1-2xe^{-x^2}-e^{-x}+xe^{-x}),\\
\left( \Delta_{\alpha}f_{(1,\alpha)}\right)'(x)
&=-2\alpha(e^x-1-2xe^{-x^2}),\\
\left( \Delta_tf_{(t,1)}\right)'(x)
&=0.
\end{align*}
The $\alpha$-level set of $\left( \Delta_{\alpha}f_{(0,\alpha)}\right)'(x)$,
$\left( \Delta_{\alpha}f_{(1,\alpha)}\right)'(x)$ and
$\left( \Delta_tf_{(t,1)}\right)'(x)$ are given in Figure~\ref{fig1a}.
It is easy to check with Theorem~\ref{th3} and Definition~\ref{def777} that
$F(x)$ is $d$-$p$-differentiable on the subinterval $[0,x_1]$
and $i$-$p$-differentiable on the subinterval $[x_2,1]$,
where $x_1\approx 0.6103$ and $x_2\approx 0.7106$,
while it is not $p$-differentiable on the subinterval $[x_1,x_2]$ (see Figure~\ref{fig1a}).
However, $F(x)$ is $gp$-differentiable with
{\small
\begin{align*}
&[F'_{gp}(x)]_{\alpha}\\
&=\Big[\inf_{\beta\geq\alpha}\min_{\textbf{t}}\left((1-x)e^{-x}+\beta^2(1-2xe^{-x^2}-e^{-x}+xe^{-x})
+t(1-\beta^2)(-4xe^{-x^2}+e^{x}-e^{-x}+xe^{-x})\right),\\
&\quad \sup_{\beta\geq\alpha}\max_{\textbf{t}}\left((1-x)e^{-x}+\beta^2(1-2xe^{-x^2}-e^{-x}+xe^{-x})
+t(1-\beta^2)(-4xe^{-x^2}+e^{x}-e^{-x}+xe^{-x})\right)\Big]\\
&=
\begin{cases}
[\underset{\beta\geq\alpha}\inf(1-2xe^{-x^2}+(1-\beta^2)(e^x-1-2xe^{-x^2})), \underset{\beta\geq\alpha}\sup((1-x)e^{-x}+\beta^2(1-2xe^{-x^2}-e^{-x}+xe^{-x}))], & 0\leq x\leq 0.6367\\
[\underset{\beta\geq\alpha}\inf((1-x)e^{-x}+\beta^2(1-2xe^{-x^2}-e^{-x}+xe^{-x})),  \underset{\beta\geq\alpha}\sup(1-2xe^{-x^2}+(1-\beta^2)(e^x-1-2xe^{-x^2}))], & 0.6367\leq x\leq 1
\end{cases}\\
&=
\begin{cases}
\left[1-2xe^{-x^2}+(1-\alpha^2)(e^x-1-2xe^{-x^2}),
(1-x)e^{-x}+\alpha^2(1-2xe^{-x^2}-e^{-x}+xe^{-x}) \right], & 0\leq x\leq 0.6103\\
[1-2xe^{-x^2}, (1-x)e^{-x}+\alpha^2(1-2xe^{-x^2}-e^{-x}+xe^{-x})], & 0.6103\leq x\leq 0.6367\\
[1-2xe^{-x^2}, 1-2xe^{-x^2}+(1-\alpha^2)(e^x-1-2xe^{-x^2}) ], & 0.6367\leq x\leq 0.7106\\
[(1-x)e^{-x}+\alpha^2(1-2xe^{-x^2}-e^{-x}+xe^{-x}),
1-2xe^{-x^2}+(1-\alpha^2)(e^x-1-2xe^{-x^2})], & 0.7106\leq x\leq 1.
\end{cases}
\end{align*}}
The $\alpha$-level set of $F'_{gp}(x)$ is shown in Figure~\ref{fig11}.
\begin{figure}
\centering
{\includegraphics[width=7.7cm,height=4.8cm]{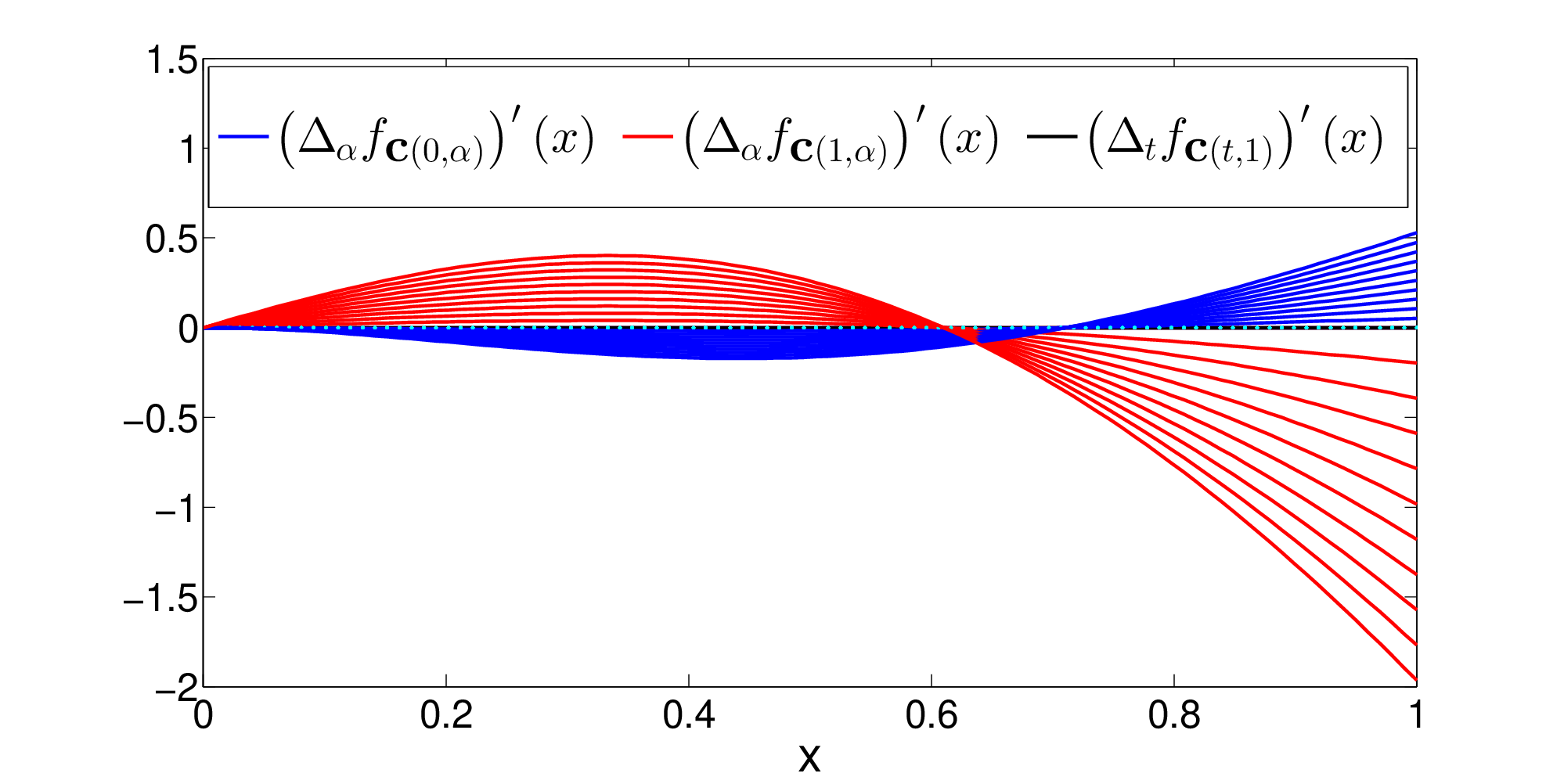}\label{fig1a}}
{\includegraphics[width=7.7cm,height=4.8cm]{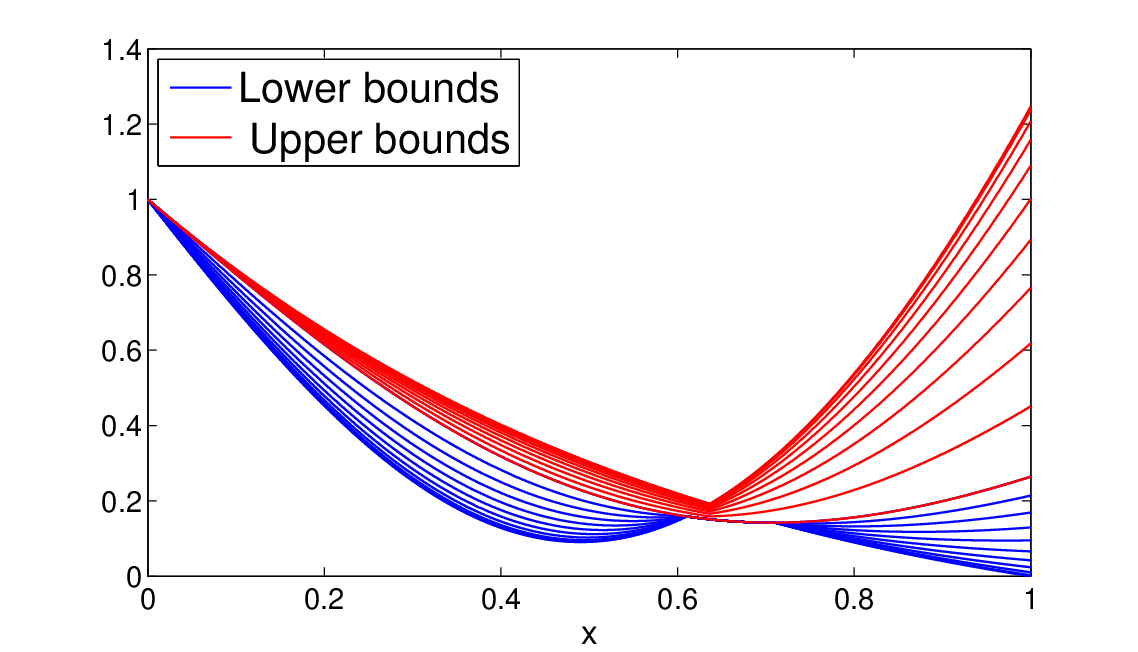}\label{fig11}}
\caption{Example~\ref{exam3}, (a) derivatives of the $\alpha$-level
set, (b) $\alpha$-level set of the $gp$-derivative.} \label{fig5}
\end{figure}
\end{example}


\subsection{Integration of fuzzy valued functions}

The integral of a fuzzy valued function
$F_{\textbf{C}_{\nu}^k}: [a,b] \rightarrow \mathcal{F}(\mathbb{R})$ with
$[F_{\textbf{C}_{\nu}^k}(x)]_{\alpha}
=\{f_{\textbf{c}(\textbf{t}, \alpha)} (x) \mid f_{\textbf{c}(\textbf{t}, \alpha)}
: [a,b] \rightarrow \mathbb{R},
~\textbf{c}(\textbf{t}, \alpha)\in [\textbf{C}_{\nu}^k]_\alpha \}$,
is defined level-wise by
\begin{equation}
\label{eq4}
\left[\int_a^bF_{\textbf{C}_{\nu}^k}(x)dx\right]_{\alpha}
=\left\{\int_a^bf_{\textbf{c}(\textbf{t}, \alpha)} (x)dx \Big{|}
f_{\textbf{c}(\textbf{t}, \alpha)}:[a,b]\rightarrow \mathbb{R}
~~\text{is integrable w.r.t. $x$ for every}~\textbf{c}(\textbf{t}, \alpha)
\in[\textbf{C}_{\nu}^k]_{\alpha} \right \}.
\end{equation}
The same comment of Remark~\ref{remk2} holds for the concept of
definite integral of a fuzzy valued function. To gain a better
understanding of this issue, we give the following example.

\begin{example}
Consider function $g$ defined by $g(x)=1-x$ and
the triangular fuzzy number $A=(0,1,2)$.
Then, based on the parametric representation \eqref{eq3},
\begin{align*}
\left[\int_0^3 A \cdot g(x)dx\right]_{\alpha}
&=\left\{\int_0^3(\alpha+t(2-2\alpha))(1-x)dx \Big{|}~ t\in[0,1] \right\}\\
&=\left\{ -\frac{3}{2}\alpha+t(-3+3\alpha)|~ t\in[0,1] \right\}\\
&=\left[-3+\frac{3}{2}\alpha, -\frac{3}{2}\alpha\right],
\end{align*}
while using the parametric representation \eqref{equ1} we obtain
\begin{align*}
&\left[\int_0^3A \cdot g(x)dx\right]_{\alpha}\\
&=\left\{\int_0^1\alpha(1-x)+t(2-2\alpha)(1-x)dx+\int_1^3(2-\alpha)(1-x)
+t(2\alpha-2)(1-x)dx \Big{|}~ t\in[0,1] \right\}\\
&=\left\{ -4+\frac{5}{2}\alpha+t(5-5\alpha)|~ t\in[0,1] \right\}\\
&=\left[-4+\frac{5}{2}\alpha,1-\frac{5}{2}\alpha\right].
\end{align*}
\end{example}

\begin{proposition}
Let ${F_{\textbf{C}_{\nu}^k}: [a,b]\rightarrow \mathcal{F}(\mathbb{R})}$
be a fuzzy valued function and
$[F_{\textbf{C}_{\nu}^k}(x)]_{\alpha} = \{ f_{\textbf{c}(\textbf{t}, \alpha)}(x)|$\\
$ f_{\textbf{c}(\textbf{t}, \alpha)}: [a,b] \rightarrow
\mathbb{R}, \textbf{c}(\textbf{t}, \alpha) \in [\textbf{C}_{\nu}^k]_{\alpha} \}$.
If $f_{\textbf{c}(\textbf{t}, \alpha)}(x)$ is integrable w.r.t. $x$,
then $F_{\textbf{C}_{\nu}^k}$ is integrable and
\begin{equation*}
\left[\int_a^bF_{\textbf{C}_{\nu}^k}(x)dx\right]_{\alpha}
=\left[\underset{\textbf{t}}{\min}\int_a^bf_{\textbf{c}(\textbf{t}, \alpha)}(x)dx,
\underset{\textbf{t}}{\max}\int_a^bf_{\textbf{c}(\textbf{t}, \alpha)}(x)dx\right].
\end{equation*}
\end{proposition}

\begin{proof}
It follows from \eqref{eq4}.
\end{proof}

\begin{proposition}
Continuous functions $F: {\cal S}\subseteq\mathbb{R}\rightarrow
\mathcal{F}(\mathbb{R})$ are integrable.
\end{proposition}

\begin{proof}
Follows immediately from Proposition~\ref{prop4}.
\end{proof}

The integral satisfies the following properties.

\begin{proposition}
\label{prop41}
Let $F(x)$ and $G(x)$ be two integrable fuzzy valued functions. Then,\\
$(1)~\int_{\beta}^{\gamma}F(x)~dx=\int_{\beta}^{\lambda}F(x)~dx
+\int_{\lambda}^{\gamma}F(x)~dx,~~~~\beta \leq \lambda \leq \gamma$.\\
$(2) \int_{\beta}^{\gamma} (a \cdot F(x)+b\cdot G(x))~dx
=a\cdot \int_{\beta}^{\gamma} F(x)~dx
+b\cdot \int_{\beta}^{\gamma} G(x)~dx,
~~~~a,b\in \mathbb{R}$.
\end{proposition}

\begin{proposition}
\label{prop6}
Let $F_{\textbf{C}_{\nu}^k}:[a,b] \rightarrow \mathcal{F}(\mathbb{R})$
be continuous. Then,
\begin{description}
\item[$(1)$] function $G_{\textbf{D}_{\nu}^n}(x)=\int_a^x F_{\textbf{C}_{\nu}^k}(u)~du$
is $p$-differentiable and $(G_{\textbf{D}_{\nu}^n})'_p(x)=F_{\textbf{C}_{\nu}^k}(x)$;
\item[$(2)$] function $H_{\textbf{E}_{\nu}^m}(x)=\int_x^b F_{\textbf{C}_{\nu}^k}(u)~du$
is $p$-differentiable and $(H_{\textbf{E}_{\nu}^m})'_p(x)=-F_{\textbf{C}_{\nu}^k}(x)$.
\end{description}
\end{proposition}

\begin{proof}
Suppose that $[F_{\textbf{C}_{\nu}^k}(x)]_{\alpha} = \{ f_{\textbf{c}(\textbf{t}, \alpha)}(x)|
f_{\textbf{c}(\textbf{t}, \alpha)}: [a,b] \rightarrow \mathbb{R}, \textbf{c}(\textbf{t}, \alpha)
\in [\textbf{C}_{\nu}^k]_{\alpha} \}$. Then,
\begin{equation*}
[G_{\textbf{D}_{\nu}^n}(x)]_{\alpha}=\left[\int_a^x F_{\textbf{C}_{\nu}^k}(u)du\right]_{\alpha}
=\left\{ \int_a^x f_{\textbf{c}(\textbf{t},\alpha)}(u)du \Big{|} f_{\textbf{c}(\textbf{t}, \alpha)}
: [a,b] \rightarrow \mathbb{R}, \textbf{c}(\textbf{t}, \alpha) \in [\textbf{C}_{\nu}^k]_{\alpha}\right\}.
\end{equation*}
From Proposition~\ref{prop4}, $f_{\textbf{c}(\textbf{t},\alpha)}$ is continuous. Then,
$\int_a^xf_{\textbf{c}(\textbf{t},\alpha)}(u)du$ is differentiable for each
$\textbf{t}\in[0,1]^k$ and $\alpha\in[0,1]$. Additionally, let us observe that if
$f_{\textbf{c}(\textbf{t},\alpha)}$ satisfies the stacking theorem, then
function $\int_a^xf_{\textbf{c}(\textbf{t},\alpha)}(u)du$ fulfills the same property.
Therefore, using Proposition~\ref{prop2}, $G_{\textbf{D}_{\nu}^n}(x)$ is $p$-differentiable and
\begin{align*}
[(G_{\textbf{D}_{\nu}^n})'_p(x)]_{\alpha}
&=\left\{\left(\int_a^x f_{\textbf{c}(\textbf{t},\alpha)}(u)du\right)' \Big{|}
f_{\textbf{c}(\textbf{t}, \alpha)}: [a,b] \rightarrow \mathbb{R},
\textbf{c}(\textbf{t}, \alpha) \in [\textbf{C}_{\nu}^k]_{\alpha}\right\}\\
&=\left\{ f_{\textbf{c}(\textbf{t},\alpha)}(x) | f_{\textbf{c}(\textbf{t}, \alpha)} :
[a,b] \rightarrow \mathbb{R}, \textbf{c}(\textbf{t}, \alpha)
\in [\textbf{C}_{\nu}^k]_{\alpha}\right\}\\
&=[F_{\textbf{C}_{\nu}^k}(x)]_{\alpha}.
\end{align*}
The second part can be proved similarly.
\end{proof}

Furthermore, using Theorem~\ref{th3} and Remark~\ref{remk1},
Proposition~\ref{prop6} holds for the fuzzy valued function $F:
{\cal S}\subseteq\mathbb{R} \rightarrow \mathcal{F}(\mathbb{R})$
with $\alpha$-level set
$\left\{f^-_{\alpha}(x)+t(f^+_{\alpha}(x)-f^-_{\alpha}(x))|
t\in[0,1] \right\}$.

\begin{proposition}
\label{prop5}
If $F$ is $p$-differentiable with no switching point
in the interval $[a,b]$, then
\begin{equation*}
\int_a^b F'_p(x)dx=F(b)\ominus_pF(a).
\end{equation*}
\end{proposition}

\begin{proof}
Because there is no switching point, by Definition~\ref{def77}
$F$ is $i$-$p$-differentiable or $d$-$p$-differentiable in the interval $[a,b]$.
Let $F$ be $i$-$p$-differentiable (the proof
for the $d$-$p$-differentiable case is similar). Then,
\begin{align*}
\left[\int_a^bF'_p(x)dx\right]_{\alpha}&=\left\{ \int_a^b\left( (f_{\alpha}^-)'(x)
+t\left((f_{\alpha}^+)'(x)-(f_{\alpha}^-)'(x)\right) \right)dx \Big{|} t\in[0,1] \right\}\\
&=\left\{ f_{\alpha}^-(b)-f_{\alpha}^-(a)+t\left(f_{\alpha}^+(b)-f_{\alpha}^+(a)
-f_{\alpha}^-(b)+f_{\alpha}^-(a)\right) \big{|} t\in[0,1] \right\}\\
&=\left\{ f_{\alpha}^-(b)+t\left(f_{\alpha}^+(b)-f_{\alpha}^-(b)\right)
-\left( f_{\alpha}^-(a)+t\left(f_{\alpha}^+(a)
-f_{\alpha}^-(a)\right) \right) \big{|} t\in[0,1] \right\}\\
&=\left[ F(b)\ominus_pF(a)\right]_{\alpha}.
\end{align*}
Thus, the proof is complete.
\end{proof}

\begin{theorem}
Let $F$ with $[F(x)]_{\alpha}=\{ f^-_{\alpha}(x)+t(f^+_{\alpha}(x)
-f^-_{\alpha}(x)) | t\in[0,1]\}$ be $p$-differentiable with $n$ switching
points at $d_i,~i=1,2,\cdots, n,~a=d_0<d_1<d_2<\cdots <d_n<d_{n+1}=b$
and exactly at these points. Then,
\begin{equation*}
F(b)\ominus_pF(a)=\sum_{i=1}^{n}\left[ \int_{d_{i-1}}^{d_i}F'_p(x)dx
\ominus_p (-1)\int_{d_i}^{d_{i+1}}F'_p(x)dx \right].
\end{equation*}
Moreover,
\begin{equation*}
\int_a^bF'_p(x)dx=\sum_{i=1}^{n+1}\left( F(d_i)\ominus_pF(d_{i-1}) \right)
\end{equation*}
and if $F(d_i)$ is crisp for $i=1,2,\cdots,n$, then $\int_a^bF'_p(x)dx=F(b)-F(a)$.
\end{theorem}

\begin{proof}
Consider one switching point only. The case of a finite
number of switching points follows easily.
Let $F$ be $i$-$p$-differentiable on $[a,d]$
and $d$-$p$-differentiable on $[d,b]$. Then, by Proposition~\ref{prop5},
\begin{align*}
\int_a^dF'_p(x)dx\ominus_p(-1)\int_d^bF'_p(x)dx&=(F(d)\ominus_pF(a))
\ominus_p(-1)(F(b)\ominus_pF(d))\\
&=(F(d)\ominus_pF(a))\ominus_p(F(d)\ominus_pF(b))\\
&=F(b)\ominus_pF(a).
\end{align*}
The last equality follows from
\begin{align*}
\Big{[}(F(d)\ominus_pF(a))
&\ominus_p(F(d)\ominus_pF(b))\Big{]}_{\alpha}
=\{ \left( f^-_{\alpha}(d)+t(f^+_{\alpha}(d)-f^-_{\alpha}(d))
-f^-_{\alpha}(a)-t(f^+_{\alpha}(a)-f^-_{\alpha}(a)) \right)\\
&\quad -\left( f^-_{\alpha}(d)+t(f^+_{\alpha}(d)-f^-_{\alpha}(d))
-f^-_{\alpha}(b)-t(f^+_{\alpha}(b)-f^-_{\alpha}(b))\right) | t\in[0,1] \}\\
&\quad=\{ f^-_{\alpha}(b)-f^-_{\alpha}(a)+t(f^+_{\alpha}(b)
-f^-_{\alpha}(b)-f^+_{\alpha}(a)+f^-_{\alpha}(d)) | t\in[0,1] \}\\
&\quad=\{ (f^-_{\alpha}(b)+t(f^+_{\alpha}(b)-f^-_{\alpha}(b)))
-(f^-_{\alpha}(a)+t(f^+_{\alpha}(a)-f^-_{\alpha}(a))) | t\in[0,1] \}\\
&\quad=\Big{[}F(b)\ominus_pF(a)\Big{]}_{\alpha}.
\end{align*}
Then, by Propositions~\ref{prop41} and \ref{prop5},
\begin{align*}
\int_a^bF'_p(x)dx&=\int_a^dF'_p(x)dx+\int_d^bF'_p(x)dx\\
&=(F(d)\ominus_pF(a))+(F(b)\ominus_pF(d)).
\end{align*}
If the values $F(d_i)$ are crisp for all switching points $d_i$, $i=1,2,\ldots,n$,
then from Remark~\ref{rem22} it follows that
\begin{align*}
\int_a^bF'_p(x)dx&=\sum_{i=1}^{n+1}\left( F(d_i)\ominus_pF(d_{i-1}) \right)\\
&=\left( F(b)-F(d_n) \right) + \left( F(d_n)-F(d_{n-1}) \right)+\cdots\\
&\quad+\left( F(d_2)-F(d_1) \right)+\left( F(d_1)-F(a) \right)\\
&=F(b)-F(a).
\end{align*}
The proof is complete.
\end{proof}


\section{Application to fuzzy differential equations}
\label{sec:04}

In this section, the following fuzzy differential equation is considered:
\begin{equation}
\label{eq5}
Y'_{\textbf{D}_{\nu}^n}=F_{\textbf{C}_{\nu}^k}(x,Y_{\textbf{D}_{\nu}^n}),
~~Y_{\textbf{D}_{\nu}^n}(x_0)=A,
\end{equation}
where $F_{\textbf{C}_{\nu}^k}: [a,b]\times \mathcal{F}(\mathbb{R})
\rightarrow \mathcal{F}(\mathbb{R})$ is a given fuzzy valued function
and $A$ is a fuzzy number. Remark~\ref{rem1} gives a useful scheme
to solve the fuzzy initial value problem \eqref{eq5}. Let
\begin{align*}
[Y_{\textbf{D}_{\nu}^n}(x)]_{\alpha}&=\{ y_{\textbf{d}(\textbf{t},\alpha)}(x)|
y_{\textbf{d}(\textbf{t},\alpha)}:[a,b]\rightarrow\mathbb{R};
\textbf{d}(\textbf{t},\alpha)\in[\textbf{D}_{\nu}^n]_{\alpha} \},\\
[F_{\textbf{C}_{\nu}^k}(x,Y_{\textbf{D}_{\nu}^n}(x))]_{\alpha}
&=\left\{ f_{\textbf{c}(\textbf{t}',\alpha)}(x,y_{\textbf{d}(\textbf{t},\alpha)}(x))|
f_{\textbf{c}(\textbf{t}',\alpha)}: [a,b]\times \mathbb{R}\rightarrow\mathbb{R};
\textbf{c}(\textbf{t}',\alpha)\in[\textbf{C}_{\nu}^k]_{\alpha}\right\},
\end{align*}
and
\begin{equation*}
[A]_{\alpha}=\left\{ a(t'',\alpha)| a(t'',\alpha)
=a^-_{\alpha}+t''(a^+_{\alpha}-a^-_{\alpha}); t''\in[0,1] \right\}.
\end{equation*}
From Definition~\ref{def1}, the differential equation \eqref{eq5} can be considered as
\begin{align*}
&\{ y'_{\textbf{d}(\textbf{t},\alpha)}(x)|
y'_{\textbf{d}(\textbf{t},\alpha)}:[a,b]\rightarrow\mathbb{R};
\textbf{d}(\textbf{t},\alpha)\in[\textbf{D}_{\nu}^n]_{\alpha} \}\\
&\hspace{4cm}= \{ f_{\textbf{c}(\textbf{t}',\alpha)}(x,y_{\textbf{d}(\textbf{t},\alpha)}(x))|
f_{\textbf{c}(\textbf{t}',\alpha)}: [a,b]\times \mathbb{R}\rightarrow\mathbb{R};
\textbf{c}(\textbf{t}',\alpha)\in[\textbf{C}_{\nu}^k]_{\alpha} \},\\
&\{ y_{\textbf{d}(\textbf{t},\alpha)}(x_0)|
y_{\textbf{d}(\textbf{t},\alpha)}:[a,b]\rightarrow\mathbb{R};
\textbf{d}(\textbf{t},\alpha)\in[\textbf{D}_{\nu}^n]_{\alpha} \}\\
&\hspace{4cm}=\{ a(t'',\alpha)| a(t'',\alpha)
=a^-_{\alpha}+t''(a^+_{\alpha}-a^-_{\alpha}); t''\in[0,1] \}.
\end{align*}
Then, by Remark~\ref{rem1}, there exist $\textbf{t}'\in[0,1]^k$ and $t''\in[0,1]$ such that
\begin{equation}
\label{eq6}
y'_{\textbf{d}(\textbf{t},\alpha)}
=f_{\textbf{c}(\textbf{t}',\alpha)}(x,y_{\textbf{d}(\textbf{t},\alpha)}),
~~y_{\textbf{d}(\textbf{t},\alpha)}(x_0)=a(t'',\alpha).
\end{equation}

\begin{theorem}
\label{th4}
Let $f: [x_0,x_0+r_1] \times \overline{B}([a^-_{\alpha},a^+_{\alpha}],r_2)
\times \overline{B}([0,1]^k ,r_3)  \times \overline{B}([0,1] ,r_4)
\rightarrow \mathbb{R}$ be Lipschitz w.r.t. second, third and fourth variables, i.e.,
$$
\exists \ L_1~s.t.~ \|f(x,y,\textbf{t}',\alpha)
-f(x,w,\textbf{t}',\alpha) \| \leq L_1 \|y-w \|,
$$
$$
\exists \ L_2~s.t.~ \|f(x,y,\textbf{t}',\alpha)
-f(x,y,\textbf{s},\alpha) \| \leq L_2 \|\textbf{t}'-\textbf{s} \|,
$$
$$
\exists \ L_3~s.t.~ \|f(x,y,\textbf{t}',\alpha)
-f(x,y,\textbf{t}',\beta) \| \leq L_3 \|\alpha-\beta \|,
$$
respectively. Then, the initial value problem
\begin{equation*}
y'=f(x,y,\textbf{t}',\alpha), ~y(x_0)=a(t'',\alpha)
=a^-_{\alpha}+t''(a^+_{\alpha}-a^-_{\alpha})
\end{equation*}
has a unique solution
$$
y(x,t'', \textbf{t}',\alpha)
=\int_{x_0}^xf(u,y,\textbf{t}',\alpha)du+a(t'',\alpha).
$$
Moreover, if $f$ is continuous in $\textbf{t}'$, then the solution
$y(x,t'', \textbf{t}', \alpha)$ is continuous in $\textbf{t}'$ and $t''$.
\end{theorem}

\begin{proof}
For the first part, see Theorem~9.6 of \cite{Bedeebook}.
The second part follows immediately from the continuity of $f$
and $y(x_0)$ w.r.t. $\textbf{t}'$ and $t''$, respectively.
\end{proof}

\begin{theorem}
\label{th5}
Let $f: [x_0,x_0+r_1] \times \overline{B}([a^-_{\alpha},a^+_{\alpha}],r_2)
\times \overline{B}([0,1]^k ,r_3)  \times \overline{B}([0,1] ,r_4)
\rightarrow \mathbb{R}$. Assume $f$ is Lipschitz w.r.t.
$y$, $\textbf{t}'$ and $\alpha$, i.e., there exists
$L_1$, $L_2$ and $L_3$ such that
$$
\|f_{\textbf{c}(\textbf{t}',\alpha)}(x,y)
-f_{\textbf{c}(\textbf{t}',\alpha)}(x,w) \| \leq L_1 \|y-w \|,
$$
$$
\|f_{\textbf{c}(\textbf{t}',\alpha)}(x,y)
-f_{\textbf{c}(\textbf{s},\alpha)}(x,y) \|
\leq L_2 \|\textbf{t}'-\textbf{s} \|,
$$
$$
\|f_{\textbf{c}(\textbf{t}',\alpha)}(x,y)
-f_{\textbf{c}(\textbf{t}',\beta)}(x,y) \|
\leq L_3 \|\alpha-\beta \|.
$$
Then, the fuzzy initial value problem
\eqref{eq5} has a unique solution.
\end{theorem}

\begin{proof}
Since $f_{\textbf{c}(\textbf{t}',\alpha)}(x,y)$  satisfies the hypothesis
of Theorem~\ref{th4}, then the ordinary differential equation
corresponding to \eqref{eq5}, i.e.,
\begin{equation*}
y'_{\textbf{d}(\textbf{t},\alpha)}
=f_{\textbf{c}(\textbf{t}',\alpha)}(x,y_{\textbf{d}(\textbf{t},\alpha)}),
~~y_{\textbf{d}(\textbf{t},\alpha)}(x_0)=a(t'',\alpha),
\end{equation*}
has a unique solution
$$
y_{\textbf{d}(\textbf{t},\alpha)}(x)
=\int_{x_0}^xf_{\textbf{c}(\textbf{t}',\alpha)}(u,y)du+a(t'',\alpha).
$$
In addition, the conditions of Theorem~\ref{th22} hold for
\begin{equation*}
[F_{\textbf{C}_{\nu}^k}(x,Y_{\textbf{D}_{\nu}^n}(x))]_{\alpha}
=\left\{ f_{\textbf{c}(\textbf{t}',\alpha)}(x,y_{\textbf{d}(\textbf{t},\alpha)}(x))|
f_{\textbf{c}(\textbf{t}',\alpha)}: [a,b]\times \mathbb{R}\rightarrow\mathbb{R};
\textbf{c}(\textbf{t}',\alpha)\in[\textbf{C}_{\nu}^k]_{\alpha} \right\}
\end{equation*}
and
\begin{equation*}
[A]_{\alpha}=\left\{ a(t'',\alpha)| a(t'',\alpha)
=a^-_{\alpha}+t''(a^+_{\alpha}-a^-_{\alpha}); t''\in[0,1] \right\}.
\end{equation*}
Then, it is easy to check that
\begin{equation*}
[Y_{\textbf{D}_{\nu}^n}(x)]_{\alpha}
=\left\{ y_{\textbf{d}(\textbf{t},\alpha)}(x)| ~
y_{\textbf{d}(\textbf{t},\alpha)}(x)
=\int_{x_0}^xf_{\textbf{c}(\textbf{t}',\alpha)}(u,y)du
+a(t'',\alpha); \textbf{d}(\textbf{t},\alpha)
\in[\textbf{D}_{\nu}^n]_{\alpha} \right\}
\end{equation*}
is the $\alpha$-level of a fuzzy set according to Theorem~\ref{th2}.
Subsequently, the fuzzy initial value problem \eqref{eq5} has a unique
fuzzy solution $Y_{\textbf{D}_{\nu}^n}(x)$.
\end{proof}

Let $f_{\textbf{c}(\textbf{t}',\alpha)}(x,y_{\textbf{d}(\textbf{t},\alpha)}(x))$ be Lipschitz w.r.t.
$y$, $\textbf{t}'$ and $\alpha$. Because for every fixed $x$ and
$\alpha\in[0,1]$, $f_{\textbf{c}(\textbf{t}',\alpha)}(x,y_{\textbf{d}(\textbf{t},\alpha)}(x))$
and $y_{d(\textbf{t},\alpha)}(x_0)=a(t'',\alpha)$ are continuous functions
in $\textbf{t}'$ and $t''$, respectively, then, by Theorem~\ref{th4}, equation \eqref{eq6}
has a unique solution $y_{\textbf{d}(\textbf{t},\alpha)}(x)=y(x, t'',\textbf{t}',\alpha)$
that is a continuous function in $\textbf{t}'$ and $t''$. Therefore,
$\underset{\textbf{t}',t''}\min~{y(x,t'',\textbf{t}',\alpha)}$ and
$\underset{\textbf{t}',t''}\max~{y(x,t'',\textbf{t}',\alpha)}$ exist with
\begin{equation}
\label{eq15}
[Y_{\textbf{D}^n_{\nu}}(x)]_{\alpha}
=[\underset{\textbf{t}',t''}\min~{y(x,t'',\textbf{t}',\alpha)},
\underset{\textbf{t}',t''}\max~{y(x,t'',\textbf{t}',\alpha)}]
\end{equation}
for each $\textbf{d}(\textbf{t},\alpha)\in[\textbf{D}_{\nu}^n]_{\alpha}$.

Now we show that the fuzzy solution of \eqref{eq5} obtained by our
approach above coincides with the solution obtained by the method
of differential inclusions \cite{Chalco2013,Chalco2007}.

\begin{theorem}
\label{th6}
Let $U$ be an open set in $[0,1]^{k+2}$. Suppose that $f_{\textbf{c}(\textbf{t}',
\alpha)}(x,y_{\textbf{d}(\textbf{t},\alpha)}(x))$ is continuous and, for each
$(t'',\textbf{t}',\alpha)\in U$, there exists a unique solution
$y_{\textbf{d}(\textbf{t},\alpha)}(\cdot)=y(\cdot,t'',\textbf{t}',\alpha)$
of problem \eqref{eq6}, where $y(\cdot,t'',\textbf{t}',\alpha)$ is continuous on $U$.
Then, there exists a fuzzy solution $Y_{\textbf{D}^n_{\nu}}(x)$ of \eqref{eq5} such that
$$
Y_{\textbf{D}^n_{\nu}}(x)=\hat{L}(x,A,C^k_{\nu})
$$
for all $a\leq x\leq b$, where $\hat{L}(x,A,C^k_{\nu})$ is its fuzzy
solution via differential inclusions.
\end{theorem}

\begin{proof}
By Theorem~\ref{th5}, the fuzzy differential equation \eqref{eq5} has a unique
solution $Y_{\textbf{D}^n_{\nu}}(x)$ such that for each $\alpha\in[0,1]$
$$
[Y_{\textbf{D}^n_{\nu}}(x)]_{\alpha}=\{y_{\textbf{d}(\textbf{t},\alpha)}(x)|~
\textbf{d}(\textbf{t}, \alpha)\in[\textbf{D}_{\nu}^n]_{\alpha} \},
$$
where $y_{\textbf{d}(\textbf{t},\alpha)}(\cdot)
=y(\cdot,t'',\textbf{t}',\alpha)$ is a solution of \eqref{eq6}.
From \eqref{eq1} and \eqref{eq22},
\begin{align*}
[Y_{\textbf{D}^n_{\nu}}(x)]_{\alpha}
&=\{y(x,a,\textbf{c},\alpha)|~a\in[A]_{\alpha}, \textbf{c}\in[\textbf{C}^k_{\nu}]_{\alpha} \}\\
&=\hat{L}(x,[A]_{\alpha},[\textbf{C}^k_{\nu}]_{\alpha},\alpha)\\
&=[\hat{L}(x,A,C^k_{\nu})]_{\alpha}
\end{align*}
for any fixed $\textbf{t}'\in[0,1]^k$ and $t''\in[0,1]$, which
shows, from Definition~\ref{def1}, that
$Y_{\textbf{D}^n_{\nu}}(x)=\hat{L}(x,A,C^k_{\nu})$.
\end{proof}

\begin{corollary}
Let $U$ be an open set in $[0,1]^{k+2}$. Suppose that
$f_{\textbf{c}(\textbf{t}',\alpha)}(x,y_{\textbf{d}(\textbf{t},\alpha)}(x))$
is continuous and, for each $(t'',\textbf{t}',\alpha)\in U$, there exists
a unique solution $y_{\textbf{d}(\textbf{t},\alpha)}(\cdot)
=y(\cdot,t'',\textbf{t}',\alpha)$ of problem \eqref{eq6}
where $y(\cdot,t'',\textbf{t}',\alpha)$ is continuous on $U$.
Then, there exists a fuzzy solution $Y_{\textbf{D}^n_{\nu}}(x)$ of \eqref{eq5} such that
$$
Y_{\textbf{D}^n_{\nu}}(x)=\hat{L}(x,A,C^k_{\nu})
$$
for all $a\leq x\leq b$, where $\hat{L}(x,A,C^k_{\nu})$ is
its fuzzy solution via Zadeh's extension principle.
\end{corollary}

\begin{proof}
The result follows from our Theorem~\ref{th6} and Corollary~1
of \cite{Chalco2007}.
\end{proof}

Our approach has the advantage that allows to obtain
the main properties of ordinary differential equations in a natural way.
Additionally, one can get algorithms for obtaining a fuzzy solution
as an extension of known algorithms for ordinary differential equations.

Another approach is possible, based on the definition of $p$-derivative for fuzzy valued functions,
using the parametric representation \eqref{equ1}. In such approach, more than one solution exists.
The existence of several solutions may be an advantage when one searches for solutions
that have specific properties, such as periodic, almost periodic or asymptotically stable.
Precisely, the following fuzzy differential equation can be considered
using the $p$-derivative:
\begin{equation}
\label{eq66}
Y'=F(x,Y), ~~Y(x_0)=A,
\end{equation}
where $Y$ and $F$ are fuzzy valued functions defined in terms of their $\alpha$-level sets
\begin{align*}
[Y(x)]_{\alpha}&=\{ y_{\alpha}^-(x)+t(y_{\alpha}^+(x)-y_{\alpha}^-(x))| t\in[0,1] \},\\
[F(x,Y(x))]_{\alpha}&=\{f_{\alpha}^-(x,y_{\alpha}^-,y_{\alpha}^+)
+t(f_{\alpha}^+(x,y_{\alpha}^-,y_{\alpha}^+)-f_{\alpha}^-(x,y_{\alpha}^-,y_{\alpha}^+))| t\in[0,1]\},
\end{align*}
and $A$ is a fuzzy number with
\begin{equation*}
[A]_{\alpha}=\left\{a^-_{\alpha}+t(a^+_{\alpha}-a^-_{\alpha})| t\in[0,1] \right\}.
\end{equation*}
If we write \eqref{eq66} in terms of its $\alpha$-level set, then, from relations \eqref{eq33} and \eqref{eq333},
the two following cases can be obtained:
\begin{description}
\item[I.] $\left
\{
\begin{array}{lll}
\{ (y_{\alpha}^-)'(x)+t((y_{\alpha}^+)'(x)-(y_{\alpha}^-)'(x))| t\in[0,1] \}\\
\hspace{5cm}=\{f_{\alpha}^-(x,y_{\alpha}^-,y_{\alpha}^+)+t(f_{\alpha}^+(x,y_{\alpha}^-,y_{\alpha}^+)
-f_{\alpha}^-(x,y_{\alpha}^-,y_{\alpha}^+))| t\in[0,1]  \},\\
\{y_{\alpha}^-(x_0)+t(y_{\alpha}^+(x_0)-y_{\alpha}^-(x_0))| t\in[0,1] \}
=\{a^-_{\alpha}+t(a^+_{\alpha}-a^-_{\alpha})| t\in[0,1] \},
\end{array}
\right.$
\item[II.] $\left\{
\begin{array}{ll}
\{(y_{\alpha}^+)'(x)+t((y_{\alpha}^-)'(x)-(y_{\alpha}^+)'(x))| t\in[0,1] \}\\
\hspace{4.9cm}=\{f_{\alpha}^-(x,y_{\alpha}^-,y_{\alpha}^+)+t(f_{\alpha}^+(x,y_{\alpha}^-,y_{\alpha}^+)
-f_{\alpha}^-(x,y_{\alpha}^-,y_{\alpha}^+))| t\in[0,1]  \},\\
\{y_{\alpha}^-(x_0)+t(y_{\alpha}^+(x_0)-y_{\alpha}^-(x_0))| t\in[0,1] \}
=\{a^-_{\alpha}+t(a^+_{\alpha}-a^-_{\alpha})| t\in[0,1] \}.
\end{array}
\right.$
\end{description}
Finally, for fixed $x$, two systems can be deduced:
\begin{equation}
\label{eq12}
\left\lbrace
\begin{array}{ll}
(y_{\alpha}^-)'(x)=f_{\alpha}^-(x,y_{\alpha}^-(x),y_{\alpha}^+(x)), ~~y_{\alpha}^-(x_0)=a_{\alpha}^-\\
(y_{\alpha}^+)'(x)=f_{\alpha}^+(x,y_{\alpha}^-(x),y_{\alpha}^+(x)), ~~y_{\alpha}^+(x_0)=a_{\alpha}^+,
\end{array}\right.
\end{equation}
\begin{equation}
\label{eq13}
\left\lbrace
\begin{array}{ll}
(y_{\alpha}^-)'(x)=f_{\alpha}^+(x,y_{\alpha}^-(x),y_{\alpha}^+(x)), ~~y_{\alpha}^-(x_0)=a_{\alpha}^-\\
(y_{\alpha}^+)'(x)=f_{\alpha}^-(x,y_{\alpha}^-(x),y_{\alpha}^+(x)), ~~y_{\alpha}^+(x_0)=a_{\alpha}^+.
\end{array}\right.
\end{equation}
In this approach, the unicity of the solution is lost, but that it is an expected situation in the fuzzy context.
Nevertheless, we can speak of existence and unicity of two solutions
(one solution for each lateral derivative), as is shown in the following result.

\begin{theorem}
Let $F: [a,b]\times \mathcal{F}(\mathbb{R})\rightarrow \mathcal{F}(\mathbb{R})$
be continuous and assume that there exists a constant $L>0$ such that
$$
D(F(x,Y),F(x,Z))\leq L D(Y,Z)
$$
for all $x \in [a,b]$ and $Y, Z\in \mathcal{F}(\mathbb{R})$.
Then problem \eqref{eq66} has two possible solutions on $[a,b]$.
\end{theorem}

\begin{proof}
Follows from Theorem~6 of \cite{Chalco2008}.
\end{proof}

\begin{example}
\label{ex:4.6}
Consider the fuzzy differential equation
\begin{equation}
\label{eq16}
\left\lbrace
\begin{array}{l}
Y_{\textbf{D}_{\nu}^n}'(x)=-Y_{\textbf{D}_{\nu}^n}(x)+(-2,1,4)
\cdot \cos(x),\\
Y_{\textbf{D}_{\nu}^n}(0)=(1,2,3),~~~~~x\in[0,4].
\end{array}\right.
\end{equation}
Using the first approach, the solution of the problem is obtained
from the ordinary differential equation
\begin{align*}
y'_{\textbf{d}(\textbf{t},\alpha)}(x)
&=-y_{\textbf{d}(\textbf{t},\alpha)}(x)+(-2+3\alpha+t'(6-6\alpha))\cos(x),\\\nonumber
y_{\textbf{d}(\textbf{t},\alpha)}(0)
&=1+\alpha+t''(2-2\alpha),~~~~t',t'' \in [0,1],~x \in [0,4],
\end{align*}
with solution
$$
y_{\textbf{d}(\textbf{t},\alpha)}(x)=(-2+3\alpha+t'(6-6\alpha))\frac{\cos(x)
+\sin(x)-e^{-x}}{2}+(1+\alpha+t''(2-2\alpha))e^{-x}.
$$
Hence, the $\alpha$-level set of the solution \eqref{eq16} is
\begin{equation*}
[Y_{\textbf{D}_{\nu}^2}(x)]_{\alpha}
=\left\{ (-2+3\alpha+t'(6-6\alpha))\frac{\cos(x)
+\sin(x)-e^{-x}}{2}+(1+\alpha+t''(2-2\alpha))e^{-x}| ~ t', t''\in[0,1]\right\},
\end{equation*}
which is the $\alpha$-level set of the fuzzy valued function
\begin{equation*}
Y_{\textbf{D}_{\nu}^2}(x)=(-2,1,4)
\cdot \frac{\cos(x)+\sin(x)-e^{-x}}{2}+(1,2,3) \cdot e^{-x}.
\end{equation*}
On the other hand, from \eqref{eq15}, the lower and upper bounds of the solution is obtained as
\begin{equation*}
[Y_{\textbf{D}_{\nu}^n}(x)]_{\alpha}
= \left\lbrace
\begin{array}{ll}
~[(-2+3\alpha)\frac{\cos(x)+\sin(x)-e^{-x}}{2}+(1+\alpha)e^{-x} ,\\
\hspace{3.5cm}~(4-3\alpha)\frac{\cos(x)+\sin(x)-e^{-x}}{2}+(3-\alpha)e^{-x} ],
& x\in [0,2.2841] \\
~[(4-3\alpha)\frac{\cos(x)+\sin(x)-e^{-x}}{2}+(1+\alpha)e^{-x} ,\\
\hspace{3.5cm}~(-2+3\alpha)\frac{\cos(x)+\sin(x)-e^{-x}}{2}+(3-\alpha)e^{-x} ],
& x\in [2.2841,4],
\end{array}\right.
\end{equation*}
which is represented in Figure~\ref{fig22}.
\begin{figure}
\centering{\includegraphics[height=6cm,width=10cm]{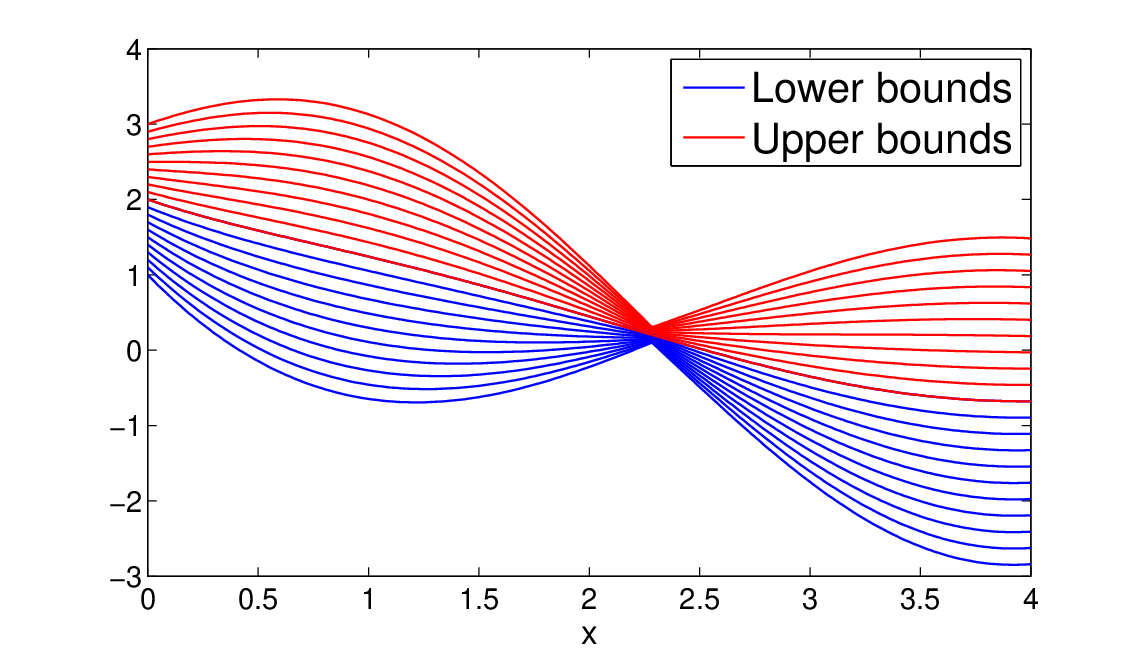}
\caption{\footnotesize{Example~\ref{ex:4.6}, lower and upper bounds
for the solution of \eqref{eq16} (first approach).}}\label{fig22}}
\end{figure}
Now, let us to use the second approach for the fuzzy differential equation \eqref{eq16}.
Systems \eqref{eq12} and \eqref{eq13} are
\begin{equation*}
\left\lbrace
\begin{array}{ll}
(y_{\alpha}^-)'(x)=\left\lbrace
\begin{array}{ll}
-y_{\alpha}^+(x)+(-2+3\alpha)\cos(x), ~~0\leq x\leq \frac{\pi}{2}\\
-y_{\alpha}^+(x)+(4-3\alpha)\cos(x), ~~~~\frac{\pi}{2}\leq x\leq 4\\
\end{array}\right.\\
(y_{\alpha}^+)'(x)=\left\lbrace
\begin{array}{ll}
-y_{\alpha}^-(x)+(4-3\alpha)\cos(x), ~~~~0\leq x\leq \frac{\pi}{2}\\
-y_{\alpha}^-(x)+(-2+3\alpha)\cos(x), ~~\frac{\pi}{2}\leq x\leq 4\\
\end{array}\right.\\
y_{\alpha}^-(0)=1+\alpha, ~~y_{\alpha}^+(0)=3-\alpha,
\end{array}\right.
\end{equation*}
and
\begin{equation*}
\left\lbrace
\begin{array}{ll}
(y_{\alpha}^-)'(x)=\left\lbrace
\begin{array}{ll}
-y_{\alpha}^-(x)+(4-3\alpha)\cos(x), ~~~~0\leq x\leq \frac{\pi}{2}\\
-y_{\alpha}^-(x)+(-2+3\alpha)\cos(x), ~~\frac{\pi}{2}\leq x\leq 4\\
\end{array}\right.\\
(y_{\alpha}^+)'(x)=\left\lbrace
\begin{array}{ll}
-y_{\alpha}^+(x)+(-2+3\alpha)\cos(x), ~~0\leq x\leq \frac{\pi}{2}\\
-y_{\alpha}^+(x)+(4-3\alpha)\cos(x), ~~~~\frac{\pi}{2}\leq x\leq 4\\
\end{array}\right.\\
y_{\alpha}^-(0)=1+\alpha, ~~y_{\alpha}^+(0)=3-\alpha,
\end{array}\right.
\end{equation*}
respectively. As stated previously, equation \eqref{eq16}
has exactly two solutions: one of them is
\begin{align*}
[Y_1(x)]_{\alpha}=[(2-1.5\alpha)\cos(x)-
&(1-1.5\alpha)\sin(x)-2.5(1-\alpha)e^x+1.5e^{-x},\\
&(-1+1.5\alpha)\cos(x)+(2-1.5\alpha)\sin(x)+2.5(1-\alpha)e^x+1.5e^{-x}],
\end{align*}
for $0\leq x\leq \frac{\pi}{2}$, and for $\frac{\pi}{2}\leq x\leq 4$ one has
\begin{align*}
[Y_1(x)]_{\alpha}=[
&(-1+1.5\alpha)\cos(x)+(2-1.5\alpha)\sin(x)
-3.1236(1-\alpha)e^x+(1.5+1.2891\times 10^{-15}\alpha)e^{-x},\\
&(2-1.5\alpha)\cos(x)-(1-1.5\alpha)\sin(x)
+3.1236(1-\alpha)e^x+(1.5+1.2891\times 10^{-15}\alpha)e^{-x}],
\end{align*}
that starts with $(i)$-$p$-differentiability, and there is no switching point on its trajectory.
The second one starts with $(d)$-$p$-differentiability and has two switching points at $x_1=0.2916$
and $x_2=2.8501$ so we have to switch to the case of $(i)$-$p$-differentiability. Therefore,
\begin{align*}
[Y_2(x)]_{\alpha}=[(2-1.5\alpha)(\cos(x)+\sin(x))-&(1-2.5\alpha)e^{-x},\\
&(-1+1.5\alpha)(\cos(x)+\sin(x))+(4-2.5\alpha)e^{-x}],
\end{align*}
$0\leq x\leq 0.2916$,
\begin{align*}
[Y_2(x)]_{\alpha}=[(2-1.5&\alpha)\cos(x)-(1-1.5\alpha)\sin(x)-0.75312(1-\alpha)e^{x}+1.5e^{-x},\\
&(-1+1.5\alpha)\cos(x)+(2-1.5\alpha)\sin(x)+0.75312(1-\alpha)e^{x}+1.5e^{-x}],
\end{align*}
$0.2916\leq x\leq \frac{\pi}{2}$,
\begin{align*}
[Y_2(x)]_{\alpha}=[(-1+1.5\alpha)(\cos(x)+&\sin(x))-(15.886-17.386\alpha)e^{-x},\\
&(2-1.5\alpha)(\cos(x)+\sin(x))+(18.886-17.386\alpha)e^{-x}],
\end{align*}
$\frac{\pi}{2}\leq x\leq 2.8501$, and
\begin{align*}
[Y_2(x)]_{\alpha}=[(-1+1.5&\alpha)\cos(x)+(2-1.5\alpha)\sin(x)-0.10803(1-\alpha)e^{x}+1.5e^{-x},\\
&(2-1.5\alpha)\cos(x)-(1-1.5\alpha)\sin(x)+0.10803(1-\alpha)e^{x}+1.5e^{-x}],
\end{align*}
$2.8501\leq x\leq 4$. The solutions are presented in Figure~\ref{fig4}.
\begin{figure}
\centering
{\includegraphics[width=0.5\linewidth]{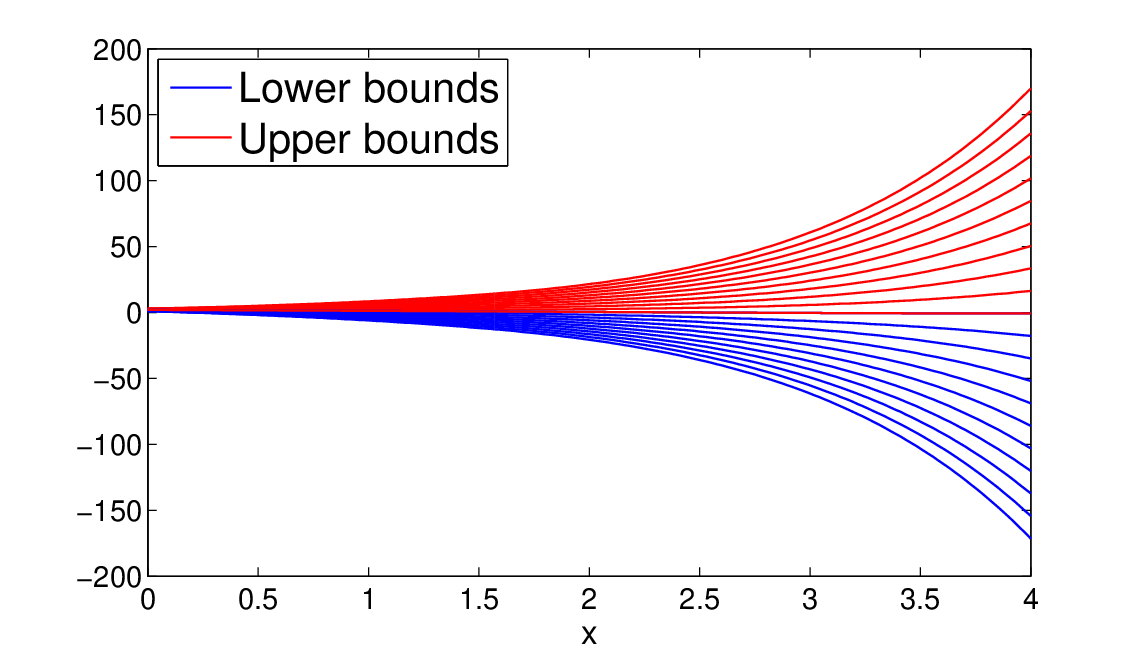}\label{fig2:sub1}}
{\includegraphics[width=0.5\linewidth]{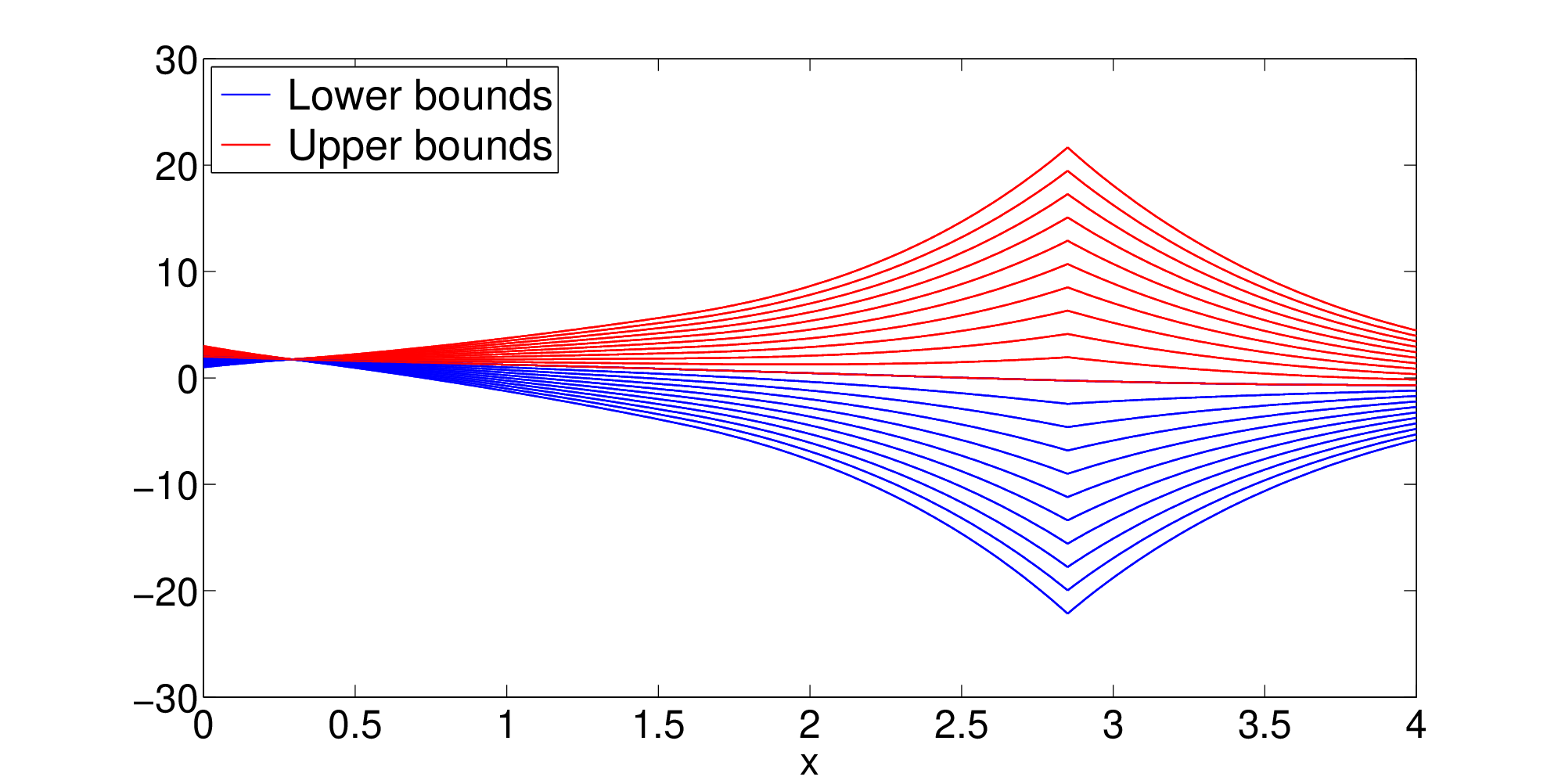}\label{fig2:sub2}}
\caption{Example~\ref{ex:4.6}, lower and upper bounds of (a)
$i$-$p$-differentiable and (b) $d$-$p$-differentiable solutions of
\eqref{eq16}.} \label{fig4}
\end{figure}
\end{example}

In our next example, we consider a simple model
for financial institution accounts and interest.

\begin{example}
\label{ex:financial}
Consider a customer who wants to make a deal with a financial institution.
He wants to invest an amount $Y_0$ by making a long-term deposit and earning
from the offered interest from the financial institution.
The financial institution offers the customer an interest rate of $5\%$ annually,
with the agreement that the money will remain deposited at the financial institution
from the first day that the capital is productive.
Additionally, deposits or withdrawals can exist in addition to capital gains.
Assume that the deposits or withdrawals take place at a constant rate $K$,
where $K$ is positive for deposits and negative for withdrawals.
Then, the rate of change of the value of the capital, $\frac{dY}{dx}$, is
\begin{equation}
\label{eq17}
\frac{dY}{dx}=0.05 \cdot Y+K, \quad Y(0)=Y_0.
\end{equation}
It is obvious that the customer wants to compare the results of different capital programs,
to see what rates of $K$ and $Y_0$ are profitable.
In this way, the parameters $K$ and $Y_0$ can be considered to be uncertain.
Depending on the nature of the uncertainty, the problem can be modeled by different
methods such as stochastic analysis, interval analysis, and the fuzzy concept.
Here we model the problem by means of fuzzy numbers by setting
$K=(-160,0,160)$ and $Y_0=(3000,3500,4000)$.

From our first approach, the corresponding ordinary differential equation is
\begin{align*}
y'_{\textbf{d}(\textbf{t},\alpha)}(x)
&=0.05y_{\textbf{d}(\textbf{t},\alpha)}(x)+(-160+160\alpha+t'(320-320\alpha)),\\\nonumber
y_{\textbf{d}(\textbf{t},\alpha)}(0)
&=3000+500\alpha+t''(1000-1000\alpha),~~~~t',t'' \in [0,1].
\end{align*}
Hence, by the first approach, the fuzzy solution
\begin{equation*}
Y_{\textbf{D}_{\nu}^2}(x)=(3000,3500,4000)\cdot e^{0.05x}+\frac{(-160,0,160)}{0.05} \cdot (e^{0.05x}-1)
\end{equation*}
is obtained, which is represented in Figure~\ref{fig6:sub1} for $50$ years.

Using our second approach, the two following systems are obtained:
\begin{equation*}
\left\lbrace
\begin{array}{ll}
(y_{\alpha}^-)'(x)=0.05y_{\alpha}^-(x)+(-160+160\alpha), & y_{\alpha}^-(0)=3000+500\alpha\\
(y_{\alpha}^+)'(x)=0.05y_{\alpha}^+(x)+(160-160\alpha), & y_{\alpha}^+(0)=4000-500\alpha,
\end{array}\right.
\end{equation*}
\begin{equation*}
\left\lbrace
\begin{array}{ll}
(y_{\alpha}^-)'(x)=0.05y_{\alpha}^+(x)+(160-160\alpha), & y_{\alpha}^-(0)=3000+500\alpha\\
(y_{\alpha}^+)'(x)=0.05y_{\alpha}^-(x)+(-160+160\alpha), & y_{\alpha}^+(0)=4000-500\alpha.
\end{array}\right.
\end{equation*}
By considering $(i)$-$p$-differentiability, the solution
\begin{align*}
[Y_1(x)]_{\alpha}=[(3700\alpha-200)e^{0.05x}-3200(\alpha-1),
~(7200-3700\alpha)e^{0.05x}+3200(\alpha-1)]
\end{align*}
is obtained, with no switching point. The second solution starts with $(d)$-$p$-differentiability
and has one switching point at $x=2.9036$, where it switches to the case of $(i)$-$p$-differentiability.
Therefore, the solution is
\begin{align*}
[Y_2(x)]_{\alpha}=[(3700\alpha-3700)e^{-0.05x}+3500e^{0.05x}&-3200(\alpha-1),\\
&(3700-3700\alpha)e^{-0.05x}+3500e^{0.05x}+3200(\alpha-1)]
\end{align*}
for $0\leq x\leq 2.9036$, and
\begin{align*}
[Y_2(x)]_{\alpha}=\left[\left(\frac{27100+102400\alpha}{37}\right)e^{0.05x}-3200(\alpha-1),
~\left(\frac{231900-102400\alpha}{37}\right)e^{0.05x}+3200(\alpha-1)\right]
\end{align*}
for $2.9036\leq x\leq50$. The solutions are shown in Figure~\ref{fig6}.
\begin{figure}
\centering
{\includegraphics[width=0.5\linewidth]{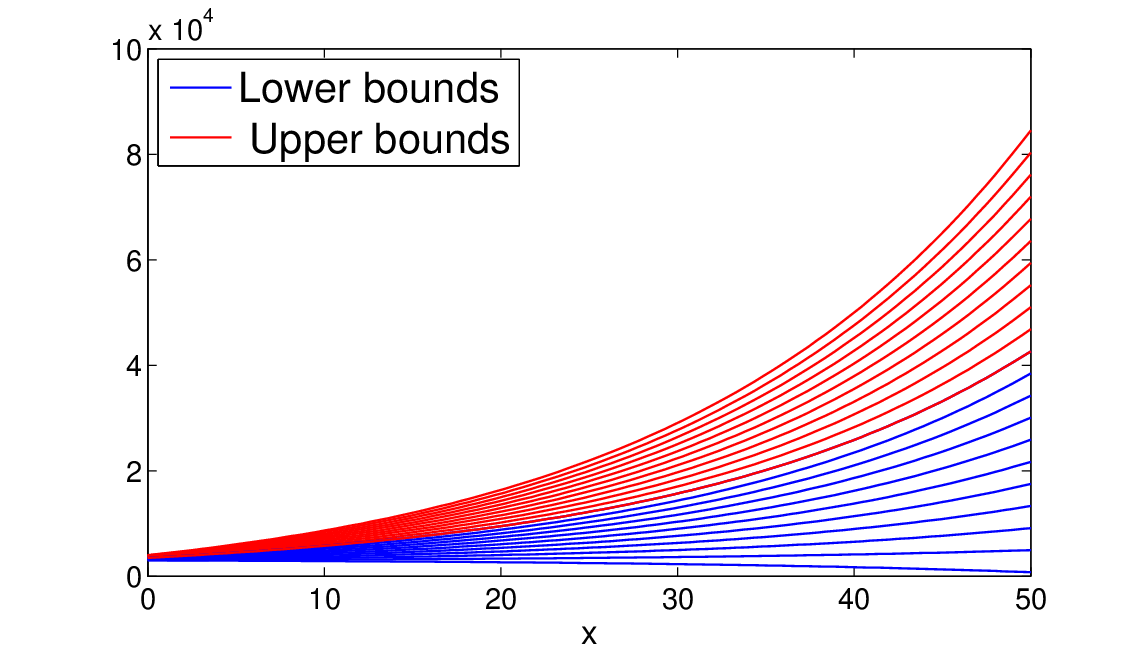}\label{fig6:sub1}}
{\includegraphics[width=0.5\linewidth]{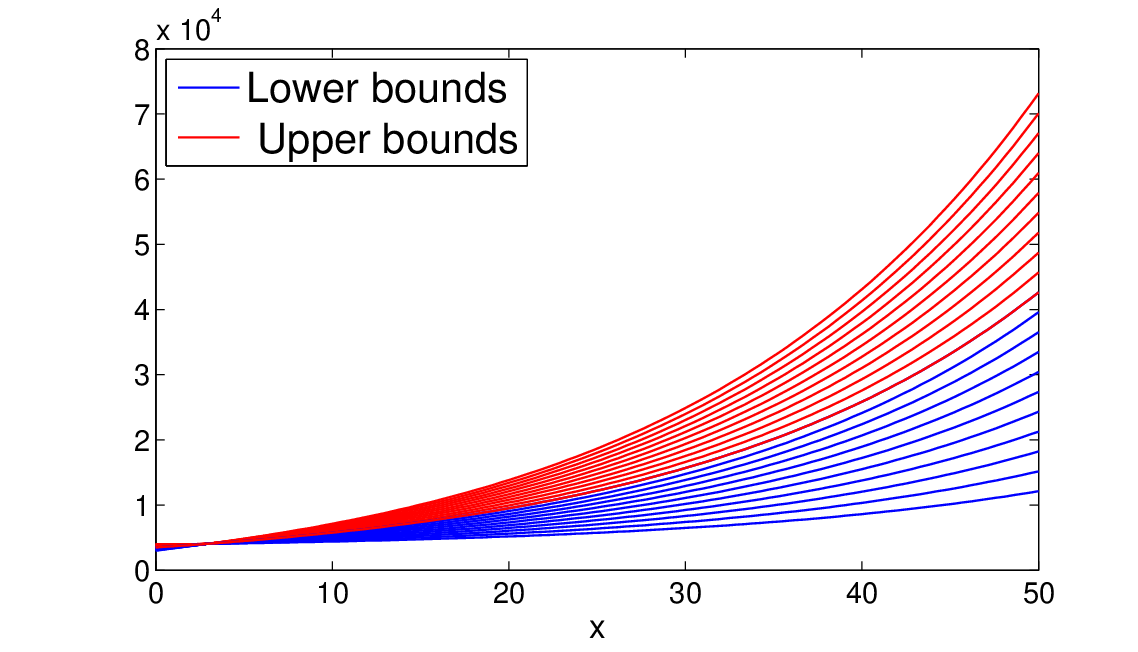}\label{fig6:sub2}}
\caption{Example~\ref{ex:financial}, lower and upper bounds of (a)
the solution by the first approach and $i$-$p$-differentiability and
(b) $d$-$p$-differentiable solution of \eqref{eq17}.} \label{fig6}
\end{figure}
\end{example}


\section{Conclusion}
\label{sec:05}

Based on parametric representations of $\alpha$-level sets,
two novel representations for fuzzy valued functions were introduced.
Using them, two analytical approaches for fuzzy differential
equations were proposed. In the first approach, the fuzzy differential equation
is converted into a crisp differential equation. Then, the transformed equation
is solved, which gives the solution to the original equation.
It is shown that this approach is related to Zadeh's extension principle
via differential inclusions, in the sense that both make use of the classical
derivative and, under certain conditions, give rise to the same solution.
In the second approach, the fuzzy differential equation originates two systems,
in such a way that two possible solutions may be attained. In this case,
the solutions are coincident with the ones obtained by using the generalized
Hukuhara differentiability concept.


\section*{Acknowledgement}

Torres is supported by FCT via project UIDB/04106/2020 (CIDMA).




\begin{thebibliography}{99}

\bibitem{Barros}
L. C. Barros, L. T. Gomes, P. A. Tonelli,
\emph{Fuzzy differential equations: an approach 
via fuzzification of the derivative operator},
Fuzzy Sets and Systems {\bf 230} (2013), 39--52.

\bibitem{Bedeebook}
B. Bede,
\emph{Mathematics of fuzzy sets and fuzzy logic},
Studies in Fuzziness and Soft Computing, 295, Springer, Heidelberg, 2013.

\bibitem{Bede2005}
B. Bede, S. G. Gal,
\emph{Generalizations of the differentiability of fuzzy-number-valued functions
with applications to fuzzy differential equations},
Fuzzy Sets and Systems {\bf 151(3)} (2005), 581--599.

\bibitem{Bede2010}
B. Bede, S. G. Gal,
\emph{Solutions of fuzzy differential equations based on generalized differentiability},
Commun. Math. Anal. {\bf 9(2)} (2010), 22--41.

\bibitem{Bede2013}
B. Bede, L. Stefanini,
\emph{Generalized differentiability of fuzzy-valued functions},
Fuzzy Sets and Systems {\bf 230} (2013), 119--141.

\bibitem{Bencsik}
A. Bencsik, B. Bede, J. Tar, J. Fodor,
\emph{Fuzzy differential equations in modeling hydraulic differential servo cylinders},
Third Romanian-Hungarian Joint Symposium on Applied Computational Intelligence,
SACI, Timisoara, Romania, 2006.

\bibitem{Buckley}
J. J. Buckley, T. Feuring,
\emph{Fuzzy differential equations},
Fuzzy Sets and Systems {\bf 110(1)} (2000), 43--54.

\bibitem{Chalco20}
Y. Chalco-Cano, W. A. Lodwick, B. Bede,
\emph{Single level constraint interval arithmetic},
Fuzzy Sets and Systems {\bf 257} (2014), 146--168.

\bibitem{Chalco2008}
Y. Chalco-Cano, H. Rom\'{a}n-Flores,
\emph{On new solutions of fuzzy differential equations},
Chaos Solitons Fractals {\bf 38(1)} (2008), 112--119.

\bibitem{Chalco2009}
Y. Chalco-Cano, H. Rom\'{a}n-Flores,
\emph{Comparison between some approaches to solve fuzzy differential equations},
Fuzzy Sets and Systems {\bf 160(11)} (2009), 1517--1527.

\bibitem{Chalco2013}
Y. Chalco-Cano, H. Rom\'{a}n-Flores,
\emph{Some remarks on fuzzy differential equations via differential inclusions},
Fuzzy Sets and Systems {\bf 230} (2013), 3--20.

\bibitem{ChalcoY}
Y. Chalco-Cano, H. Rom\'{a}n-Flores, M. Rojas-Medar, O.R. Saavedra, M.D. Jim\'{e}nez-Gamero,
\emph{The extension principle and a decomposition of fuzzy sets},
Inform. Sci. {\bf 177(23)} (2007), 5394--5403.

\bibitem{Chalco200}
Y. Chalco-Cano, A. Rufi\'{a}n-Lizana, H. Rom\'{a}n-Flores, M.D. Jim\'{e}nez-Gamero,
\emph{Calculus for interval-valued functions using generalized Hukuhara derivative and applications},
Fuzzy Sets and Systems {\bf 219} (2013), 49--67.

\bibitem{Diamond}
P. Diamond,
\emph{Time-dependent differential inclusions, cocycle attractors and fuzzy differential equations},
IEEE Trans. Fuzzy Syst. {\bf 7} (1999), 734--740.

\bibitem{Diamond2}
P. Diamond,
\emph{Stability and periodicity in fuzzy differential equations},
IEEE Trans. Fuzzy Syst. {\bf 8} (2000), 583--590.

\bibitem{Dubois}
D. Dubois, H. Prade,
\emph{Operations on fuzzy numbers},
Internat. J. Systems Sci. {\bf 9(6)} (1978), 613--626.

\bibitem{Giachetti}
R. E. Giachetti, R. E. Young,
\emph{A parametric representation of fuzzy numbers and their arithmetic operators},
Fuzzy Sets and Systems {\bf 91(2)} (1997), 185--202.

\bibitem{Goetschel}
R. Goetschel, Jr., W. Voxman,
\emph{Elementary fuzzy calculus},
Fuzzy Sets and Systems {\bf 18(1)} (1986), 31--43.

\bibitem{Gomes2015}
L.T. Gomes, L.C. Barros,
\emph{A note on the generalized difference and the generalized differentiability},
Fuzzy Sets and Systems {\bf 280} (2015), 142--145.

\bibitem{Heidari2017}
M. Heidari, M. Ramezanzadeh, A.H. Borzabadi, O.S. Fard,
\emph{Solutions to fuzzy variational problems: necessary and sufficient conditions},
Int. J. Modelling, Identification and Control {\bf 28} (2017), 187--198.

\bibitem{Heidari2016}
M. Heidari, M.R. Zadeh, O.S. Fard, A.H. Borzabadi,
\emph{On unconstrained fuzzy-valued optimization problems},
Int. J. Fuzzy Syst. {\bf 18(2)} (2016), 270--283.

\bibitem{Huang}
H. Huang, C. Wu,
\emph{Approximation of fuzzy functions by regular fuzzy neural networks},
Fuzzy Sets and Systems {\bf 177} (2011), 60--79.

\bibitem{Hullermeier}
E. H\"{u}llermeier,
\emph{An approach to modelling and simulation of uncertain dynamical systems},
Internat. J. Uncertain. Fuzziness Knowledge-Based Systems
{\bf 5(2)} (1997), 117--137.

\bibitem{Kaleva}
O. Kaleva,
\emph{Fuzzy differential equations},
Fuzzy Sets and Systems {\bf 24(3)} (1987), 301--317.

\bibitem{Klir2}
G. J. Klir,
\emph{Fuzzy arithmetic with requisite constraints},
Fuzzy Sets and Systems {\bf 91} (1997), 165--175.

\bibitem{klir}
G. J. Klir, B. Yuan,
\emph{Fuzzy sets and fuzzy logic},
Prentice Hall PTR, Upper Saddle River, NJ, 1995.

\bibitem{Lodwick}
W.A. Lodwick,D. Dubois,
\emph{Interval linear systems as a necessary step in fuzzy linear systems},
Fuzzy Sets and Systems {\bf 281} (2015), 227--251.

\bibitem{Lupulescu}
V. Lupulescu,
\emph{On a class of fuzzy functional differential equations},
Fuzzy Sets and Systems {\bf 160(11)} (2009), 547--1562.

\bibitem{Melliani}
S. Melliani,
Semi linear equation with fuzzy parameters,
Notes IFS {\bf 5} (1999), no.~4, 42--47.

\bibitem{Chalco2007}
M.T. Mizukoshi, L.C. Barros, Y. Chalco-Cano, H. Rom\'{a}n-Flores, R.C. Bassanezi,
\emph{Fuzzy differential equations and the extension principle},
Inform. Sci. {\bf 177(17)} (2007), 3627--3635.

\bibitem{Negoita}
C.V. Negoi\c t\u a\,\ D.A. Ralescu,
\emph{ Applications of fuzzy sets to systems analysis},
John Wiley \& Sons, New York, 1975.

\bibitem{Oberguggenberger}
M. Oberguggenberger, S. Pittschmann,
\emph{Differential equations with fuzzy parameters},
Math. Modelling Syst. {\bf 5} (1999), 181--202.

\bibitem{Puri}
M. L. Puri, D.A. Ralescu,
\emph{Differentials of fuzzy functions},
J. Math. Anal. Appl. {\bf 91(2)} (1983), 552--558.

\bibitem{Ramezanadeh}
M. Ramezanadeh, M. Heidari, O.S. Fard, A.H. Borzabadi,
\emph{On the interval differential equation: novel solution methodology},
Adv. Difference Equ. {\bf 2015} (2015), , 23.

\bibitem{Stefanini}
L. Stefanini,
\emph{A generalization of Hukuhara difference and division for interval and fuzzy arithmetic},
Fuzzy Sets and Systems {\bf 161(11)} (2010), 1564--1584.

\bibitem{Stefanini2}
L. Stefanini, B. Bede,
\emph{Generalized Hukuhara differentiability of interval-valued functions and interval differential equations},
Nonlinear Anal. {\bf 71} (2009), 1311--1328.

\bibitem{Stefaninia2006}
L. Stefanini, L. Sorini, M.L. Guerra,
\emph{Parametric representation of fuzzy numbers and application to fuzzy calculus},
Fuzzy Sets and Systems {\bf 157(18)} (2006),  2423--2455.

\bibitem{Zadeh}
L. A. Zadeh,
\emph{The concept of a linguistic variable and its application to approximate reasoning. I},
Information Sci. {\bf 8} (1975), 199--249.

\bibitem{Zadeh1994}
L. A. Zadeh,
\emph{The role of fuzzy logic in modeling, identification and control},
MIC---Model. Identif. Control {\bf 15(3)} (1994),  191--203.

\bibitem{Zadeh2}
L. A. Zadeh,
\emph{Toward a generalized theory of uncertainty (GTU)---an outline},
Inform. Sci. {\bf 172} (2005), 1--40.

\end{thebibliography}
\end{document}